\documentclass[a4paper,11pt]{amsart}

\usepackage{mathpazo}
\usepackage{amssymb}
\usepackage[pagebackref,%colorlinks,linkcolor=black,citecolor=black%
    ,pdfborder={0 0 0}
    ,urlcolor=black,a4paper,hypertexnames=false]{hyperref}
\hypersetup{pdfauthor={Daniel Fauser, Clara L\"oh},pdftitle=Variations
on the theme of the UBC}

\usepackage{pgf}
\usepackage{tikz}
\usetikzlibrary{decorations.pathmorphing}
\usepackage[all]{xy}
\SelectTips{cm}{}

\newtheorem{thm}{Theorem}[section]
\newtheorem{prop}[thm]{Proposition}
\newtheorem{lem}[thm]{Lemma}
\newtheorem{cor}[thm]{Corollary}

\newtheorem{question}[thm]{Question}

\theoremstyle{remark}
\newtheorem{rem}[thm]{Remark}

\newtheorem{notation}[thm]{Notation}

\theoremstyle{definition}
\newtheorem{defi}[thm]{Definition}

\usepackage{amsfonts}
\newcommand{\Z}{\mathbb{Z}}
\newcommand{\Q}{\mathbb{Q}}
\newcommand{\R}{\mathbb{R}}
\newcommand{\N}{\mathbb{N}}

\DeclareMathOperator{\id}{id}

\DeclareMathOperator{\map}{map}

\def\epsilon{\varepsilon}
\DeclareMathOperator{\res}{res}

\DeclareMathOperator{\ubc}{UBC}
\def\UBC#1{%
  $#1$-$\ubc$}

\def\sv#1{\|#1\|}

\makeatletter
\newcommand\norm{\bBigg@{0.8}}

\makeatother
\newcommand{\ifsv}[2][norm]{\csname #1l\endcsname\bracevert\!#2\!%
                            \csname #1r\endcsname\bracevert}
\newcommand{\stisv}[2][norm]{\indnorml[#1]{#2}{\Z}{\infty}}
\newcommand{\indnorml}[4][flex]{\csname #1l\endcsname\|#2%
                                 \csname #1r\endcsname\|_{#3}^{#4}\mathclose{}}

\def\actson{\curvearrowright}

\def\ucov#1{%
  \widetilde{#1}
}

\usepackage{color}
\usepackage{pdfcolmk}

\author{Daniel Fauser}
\author{Clara L\"oh}
\title[Variations on the theme of the UBC]%
      {Variations on the theme of\\ the uniform boundary condition}
\date{\today.\ \copyright{\ D.~Fauser, C.~L\"oh 2017}. 
    This work was supported by the CRC~1085 \emph{Higher Invariants} 
    (Universit\"at Regensburg, funded by the DFG).
    \\%\draftinfo\\
     MSC~2010 classification: 55N10, 57N65}
\begin{document}

\begin{abstract}
  The uniform boundary condition in a normed chain complex asks for a
  uniform linear bound on fillings of null-homologous cycles. For the
  $\ell^1$-norm on the singular chain complex, Matsumoto and
  Morita established a characterisation of the uniform boundary
  condition in terms of bounded cohomology. In particular, spaces with
  amenable fundamental group satisfy the uniform boundary condition in
  every degree. We will give an alternative proof of statements of this
  type, using geometric F\o lner arguments on the chain level instead
  of passing to the dual cochain complex. These geometric methods have
  the advantage that they also lead to integral refinements. In particular, 
  we obtain applications in the context of integral foliated simplicial
  volume.
\end{abstract}

\maketitle

% just for debugging:
%\tableofcontents

%%%%%%%%%%%%%%%%%%%%%%%%%%%%%%%%%%%%%%%%%%%%%%%%%%%%%%%%
\section{Introduction}

The \emph{uniform boundary condition} in a normed chain complex asks
for a uniform linear bound on fillings of null-homologous cycles~\cite{mm} 
(Definition~\ref{def:ubc}). For the $\ell^1$-norm on the singular
chain complex, Matsumoto and Morita proved a characterisation of the
uniform boundary condition in terms of bounded cohomology of the dual
cochain complex~\cite{mm}.  In particular, spaces with amenable
fundamental group satisfy the uniform boundary condition in every
degree. Efficient fillings of this sort are used in glueing
formulae for simplicial volume~\cite{vbc,kuessner} and the calculation of
simplicial volume of smooth manifolds with non-trivial smooth
$S^1$-action~\cite{yano}.

In the present article, we will give an alternative proof of the
uniform boundary condition in the presence of amenability, using
geometric F\o lner arguments on the chain level instead of passing to
the dual cochain complex. These geometric methods lead to integral
refinements (Theorems~\ref{mainthm:stisv},
\ref{mainthm:parsvab},~\ref{mainthm:parsv}).  The prototypical result
reads as follows (which is a special case of the results by Matsumoto
and Morita):

\begin{prop}[UBC for the rational $\ell^1$-norm]\label{mainprop:Q}
  Let $M$ be an aspherical topological space with amenable
  fundamental group and let $n \in \N$. Then the chain
  complex~$C_*(M;\Q)$ satisfies~$n$-$\ubc$, i.e.: There is a
  constant~$K \in \R_{>0}$ such that: If $c \in C_n(M;\Q)$ is a
  null-homologous cycle, then there exists a filling chain~$b \in
  C_{n+1}(M;\Q)$ satisfying
  \[ \partial b = c
     \quad\text{and}\quad
     |b|_1 \leq K \cdot |c|_1.
  \]
\end{prop}
Here, we call a topological space \emph{aspherical} if it is path-connected,
locally path-connected and admits a contractible universal covering.

Our method of proof is related to the F\o lner filling argument
for the vanishing of integral foliated simplicial volume
of aspherical oriented closed connected manifolds with amenable
fundamental group~\cite[Section~6]{flps}. More precisely, the
proof consists of three steps:
\begin{enumerate}
\item Lifting appropriate chains to chains on the universal covering,
  taking translations by F\o lner sets, and estimating the size of
  these translates (lifting lemma).
\item Filling these chains more efficiently (filling lemma; in this
  step, asphericity is essential).
\item Projecting the chains to the original chain complex on
  the base space and dividing by the order of the F\o lner
  sets (in this step, special properties of the coefficients are needed).
\end{enumerate}
Taking the limit along a F\o lner sequence then gives the desired
estimates.

Our statements are weaker than the original result of Matsumoto and
Morita because the method requires asphericity (or at least a highly
connected universal covering space). However, our proof is more
constructive and yields refined information in integral contexts:

\begin{thm}[UBC for the stable integral $\ell^1$-norm]
  \label{mainthm:stisv}
  Let $M$ be an aspherical topological space with countable amenable
  residually finite fundamental group and let $n \in \N$. Then there
  is a constant~$K \in \R_{>0}$ such that: If $c \in C_n(M;\Z)$ is a
  null-homologous cycle, then there is a sequence~$(b_k)_{k \in \N}$
  of chains and a sequence~$(M_k)_{k \in \N}$ of covering
  spaces of~$M$ with the following properties:
  \begin{itemize}
  \item For each~$k \in \N$, there is a regular finite-sheeted
    covering~$p_k \colon M_k \longrightarrow M$ of~$M$, and the
    covering degrees~$d_k$ satisfy~$\lim_{k \rightarrow \infty} d_k =
    \infty$.
  \item For each~$k \in \N$ we have~$b_k \in C_{n+1}(M_k;\Z)$
    and
    \[ \partial b_k = c_k
       \quad\text{and}\quad
       |b_k|_1 \leq d_k \cdot K \cdot |c|_1,
    \]
    where $c_k \in C_n(M_k;\Z)$ denotes the full $p_k$-lift of~$c$ (see Equation~(\ref{eq:fulllift})
		on p.~\pageref{eq:fulllift}).   
  \end{itemize}
\end{thm}

Dynamical versions of F\o lner sequences lead to corresponding
results for the parametrised $\ell^1$-norm:

\begin{thm}[parametrised UBC for tori]
  \label{mainthm:parsvab}
  Let $d\in\N_{>0}$ and let $M := (S^1)^d$ be the $d$-torus, let $\Gamma := \pi_1(M) \cong \Z^d$, and 
  let~$\alpha=\Gamma \curvearrowright (X,\mu)$ be an (essentially) free standard $\Gamma$-space.
	Then
	\[ C_*(M;\alpha) = L^\infty(X;\Z)\underset{\Z\Gamma}{\otimes} C_{*}(\ucov M;\Z)
	\] satisfies the uniform boundary condition in every degree, i.e.: For every~$n \in \N$
        there is a constant~$K \in \R_{>0}$ such that: For every null-homologous cycle~$c
        \in C_n(M;\alpha)$ there exists a chain~$b \in C_{n+1}(M;\alpha)$ with
        \[ \partial b = c \quad \text{and} \quad |b|_1 \leq K \cdot |c|_1.
	\]
\end{thm}

\begin{thm}[mixed UBC for the parametrised $\ell^1$-norm]
  \label{mainthm:parsv}
  Let $M$ be an aspherical topological space with amenable
  fundamental group and let $n \in \N$.
	Let~$\alpha=\Gamma \curvearrowright (X,\mu)$ be an (essentially) free standard $\Gamma$-space.
	Then there is a constant~$K \in \R_{>0}$ such that: 
	For every cycle~$c\in C_n(M;\Z) \subset C_n(M;\alpha)$ 
	that is null-homologous in~$C_*(M;\Z)$ there exists 
	a parametrised chain~$b \in C_{n+1}(M;\alpha)$ with
	\[ \partial b = c \quad \text{and} \quad |b|_1 \leq K \cdot |c|_1.
	\]
\end{thm}

%%%%%%%%%
\subsection*{Applications}

Integral foliated simplicial volume is the dynamical sibling of
simplicial volume, defined as the parametrised
$\ell^1$-semi-norm of the fundamental
class~\cite{gromovmetric,mschmidt} (see Sections~\ref{subsec:ifsv} and
\ref{subsec:ifsvrel} for definitions). The main interest in integral
foliated simplicial volume comes from the fact that this invariant
gives an upper bound for $L^2$-Betti numbers (and whence the Euler
characteristic)~\cite{mschmidt}. It is therefore natural to
investigate for which (aspherical) manifolds vanishing of integral
foliated simplicial volume is equivalent to vanishing of ordinary
simplicial volume.

Integral foliated simplicial volume of oriented closed connected
aspherical manifolds with amenable fundamental group is
trivial~\cite{flps} and oriented closed connected Seifert
$3$-manifolds with infinite fundamental group have trivial integral
foliated simplicial volume~\cite{loehpagliantini}. Triviality of
integral foliated simplicial volume is preserved under taking
cartesian products~\cite{mschmidt}, finite
coverings~\cite{loehpagliantini}, and (in the aspherical case) under
ergodic bounded measure equivalence~\cite{loehpagliantini}. Integral
foliated simplicial volume of aspherical oriented closed connected
surfaces and hyperbolic $3$-manifolds coincides with ordinary
simplicial volume~\cite{loehpagliantini}. However, for
higher-dimensional hyperbolic manifolds, integral foliated simplicial
volume is uniformly bigger than ordinary simplicial
volume~\cite{flps}.

The parametrised versions of the uniform boundary condition serve as
first step for glueing results for integral foliated simplicial
volume. We will give a simple example of such a glueing result along
tori in Section~\ref{sec:glueing}. Moreover, the uniform boundary
condition for the parametrised $\ell^1$-norm on~$S^1$ is a crucial
ingredient in the treatment of $S^1$-actions for integral foliated
simplicial volume~\cite{fauser}.

Another application of the F\o lner filling technique is that one can
reprove the vanishing of $\ell^1$-homology of 
amenable groups -- without using bounded cohomology
(Section~\ref{sec:l1homology}).

% stable integral simplicial volume of some more 3-manifolds

%%%%%%%%%%
\subsection*{Organisation of this article}

Section~\ref{sec:nch} contains a brief introduction into normed chain
complexes and the $\ell^1$-norms on singular chain complexes.  In
Section~\ref{sec:UBC}, we recall the terminology for the uniform
boundary condition and we survey the results of Matsumoto and Morita.
Section~\ref{sec:keylemmas} introduces the two key lemmas (lifting and
filling lemma) for the F\o lner filling argument.

The prototypical case of the uniform boundary condition for the
rational $\ell^1$-norm (Proposition~\ref{mainprop:Q}) is proved in
Section~\ref{sec:ubcq}.  The same proof with refined F\o lner
sequences gives Theorem~\ref{mainthm:stisv}
(Section~\ref{sec:stiubc}).  The dynamical versions
Theorem~\ref{mainthm:parsvab} and Theorem~\ref{mainthm:parsv} are
proved in Section~\ref{sec:parsvab} and Section~\ref{sec:parsv},
respectively. The uniform boundary condition on the ordinary integral
singular chain complex is briefly discussed in Section~\ref{sec:ubcz}.

Integral foliated simplicial volume of glueings along tori is considered
in Section~\ref{sec:glueing} and $\ell^1$-homology of amenable groups
is treated in Section~\ref{sec:l1homology}.

%%%%%%%%%
\subsection*{Acknowledgements}

We would like to thank the anonymous referee for carefully reading the manuscript
and providing constructive feedback.

%%%%%%%%%%%%%%%%%%%%%%%%%%%%%%%%%%%%%%%%%%%%%%%%%%%%%%%%
\section{Normed chain complexes}\label{sec:nch}

We recall basic notions in the context of normed abelian groups and
normed chain complexes.

%%%%%%%
\subsection{Normed chain complexes}

\begin{defi}[(semi-)norms on abelian groups]
  \hfil
  \begin{itemize}
    \item A \emph{semi-norm} on an abelian group~$A$  
      is a map~$|\cdot| \colon A \longrightarrow \R_{\geq 0}$ 
        with the following properties: 
        \begin{itemize}
          \item  We have~$|0| = 0$.
          \item For all~$x \in A$ we have~$|-x| = |x|$.
          \item For all~$x , y \in A$ we have~$|x + y| \leq |x| + |y|$.
        \end{itemize}
        %The semi-norm is \emph{homogeneous} if for
        %all~$n \in \Z \setminus \{0\}$ and all~$x \in A$ we have~$|n
        %\cdot x| = |n| \cdot |x|$, where~$|n|$ denotes the ordinary
        %absolute value on~$\Z$.
      \item A \emph{norm} on an abelian group~$A$ is a
        semi-norm~$|\cdot|$ on~$A$ with the property that
        $|x| = 0$ holds only for~$x = 0$.
    \item A \emph{(semi-)normed abelian group} is an abelian group 
      together with a (semi-)norm.
    \item A homomorphism~$\varphi \colon A \longrightarrow B$ between
      normed abelian groups is \emph{bounded} if there is a constant~$C \in \R_{\geq 0}$
      satisfying for all~$a \in A$ the estimate
      \[ \bigl|\varphi(a)\bigr| \leq C \cdot |a|.
      \]
      The least such constant is the \emph{norm of~$\varphi$}, denoted by~$\|\varphi\|$. 
  \end{itemize}
\end{defi}

\begin{defi}[normed chain complex, induced semi-norm on homology]
    A \emph{normed} chain complex is a chain complex in the category
    of normed abelian groups (with bounded homomorphisms as morphisms).
    Let $C_*$ be a normed chain complex and let $n \in \N$. Then the
    norm~$|\cdot|$ on~$C_n$ induces a semi-norm~$\|\cdot\|$ on~$H_n(C_*)$
    via
    \[ \| \alpha\| := \inf \bigl\{ |c| \bigm| c \in C_n,\ \partial c = 0,\ [c] = \alpha \in H_n(C_*) \bigr\} \in \R_{\geq 0}
    \]
    for all~$\alpha \in H_n(C_*)$.
\end{defi}

We will also need the corresponding equivariant versions:

\begin{defi}[twisted normed modules]
  Let $\Gamma$ be a group. A \emph{normed $\Gamma$-module} is a normed
  abelian group together with an isometric $\Gamma$-action. 
\end{defi}

%%%%%%%%
\subsection{The $\ell^1$-norm on the singular chain complex}

A geometrically interesting example of a normed chain complex
is given by the singular chain complex:

\begin{defi}[the twisted $\ell^1$-norm]
  Let $M$ be path-connected, locally path-connected topological space
  that admits a universal covering~$\ucov M$ (e.g., a connected
  CW-complex). Let $\Gamma := \pi_1(M)$, and let $A$ be a normed right
  $\Z \Gamma$-module. For $n \in \N$, we define the \emph{twisted $\ell^1$-norm} 
  \begin{align*}
    | \cdot |_1 \colon C_n(M;A) & \longrightarrow \R_{\geq 0} \\
    \sum_{j=1}^m f_j \otimes \sigma_j & \longmapsto \sum_{j=1}^m |f_j|
  \end{align*}
  on~$C_n(M;A) := A \otimes_{\Z \Gamma} C_n(\ucov M;\Z)$, where
  $C_n(\ucov M;\Z)$ carries the left $\Z \Gamma$-module structure
  induced by the deck transformation action of~$\Gamma$ on~$\ucov M$.
  Here, we assume that $\sum_{j=1}^m f_j \otimes \sigma_j$ is in
  \emph{reduced form}, i.e., that the singular simplices $\sigma_1,\dots,\sigma_m \in
  \map(\Delta^n,\ucov M)$ all belong to different $\Gamma$-orbits.
\end{defi}
% is well-defined, i.e., different reduced forms all lead to the
% same norm!

In the situation of the previous definition, $C_*(M;A)$ is a normed
chain complex with respect to the twisted $\ell^1$-semi-norm. We denote
the induced \emph{twisted $\ell^1$-semi-norm} on~$H_*(M;A)$ by~$\|\cdot\|_{1,A}$.

Using the $\ell^1$-semi-norm on singular homology, we can define (twisted)
simplicial volumes:

\begin{defi}[simplicial volume]
  Let $M$ be an oriented closed connected $n$-manifold with
  fundamental group~$\Gamma$ and let $A$ be a normed $\Z\Gamma$-module
  together with a $\Z \Gamma$-homomorphism~$i \colon \Z \longrightarrow A$ 
	(where we consider~$\Z$ as $\Z \Gamma$-module with respect to the trivial $\Gamma$-action). 
	Then the \emph{$A$-simplicial volume of~$M$} is defined by
  \[ \bigl\| M \|_{A}
     := \bigl\| [M]_A \bigr\|_{1,A} \in \R_{\geq 0}, 
  \]
  where $[M]_A \in H_n(M;A)$ denotes the push-forward of the integral fundamental
  class~$[M]_\Z \in H_n(M;\Z)$ along~$i$.
\end{defi}

For example, the real numbers~$\R$ with the standard norm and the
canonical inclusion~$\Z \longrightarrow \R$ 
(of~$\Z \Gamma$-modules with trivial $\Gamma$-action)
lead to the classical simplicial volume by Gromov~\cite{vbc}.

%For aspherical spaces, we can also make use of the
%following fact~\cite[Proposition~5.11]{loehpagliantini}:
% also describe the map? say something about the mapping theorem?!
%\begin{prop}[$\ell^1$-semi-norm via group homology]
%  Let $M$ be an aspherical topological space and
%  let~$\widetilde M$ be its universal covering, let $\Gamma := \pi_1(M)$,
%  and let $A$ be a normed $\Z \Gamma$-module. Then there is 
%  a natural chain map~$C_*(M;A) \longrightarrow C_*(\Gamma;A)$
%  that induces an isometric isomorphism
%  \[ H_*(M;A) \longrightarrow H_*(\Gamma;A)
%  \]
%  with respect to the $\ell^1$-semi-norms.
%\end{prop}

%%%%%%%%
\subsection{The parametrised $\ell^1$-norm}\label{subsec:ifsv}

We will now focus on twisted $\ell^1$-norms where the coefficients are
induced from actions of the fundamental group on probability spaces.
Twisted $\ell^1$-norms of this type lead to integral foliated
simplicial volume~\cite{gromovmetric,mschmidt}, a notion also
studied in Section~\ref{sec:glueing}.

\begin{defi}[standard $\Gamma$-space]
  Let $\Gamma$ be a countable group. A \emph{standard $\Gamma$-space} is
  a standard Borel probability space~$(X,\mu)$ together with a measurable
  probability measure preserving left $\Gamma$-action on~$(X,\mu)$. 
\end{defi}

Every countable group~$\Gamma$ admits an essentially free ergodic
standard $\Gamma$-space, for instance the Bernoulli
shift~\cite{mschmidt}.

\begin{defi}[parametrised $\ell^1$-norm]
  Let $M$ be a path-connected, locally path-connected topological
  space that admits a universal covering space~$\ucov M$, let $\Gamma :=
  \pi_1(M)$, and let $\alpha = \Gamma \actson (X,\mu)$ be a standard
  $\Gamma$-space. Then $L^\infty(X,\mu;\Z)$ together with the 1-norm
  \begin{align*}
    L^\infty(X,\mu;\Z) & \longrightarrow \R \\
    f & \longmapsto \int_X |f| \;d\mu
  \end{align*}
  and the canonical right $\Gamma$-action is a normed $\Z
  \Gamma$-module, which we also denote by~$L^\infty(X;\Z)$ or
  $L^\infty(\alpha;\Z)$. 
  The associated twisted $\ell^1$-norm on
  \[ C_*(M;\alpha) := L^\infty(X;\Z) \otimes_{\Z \Gamma} C_*(\ucov M;\Z)
  \]
  is the \emph{$\alpha$-parametrised $\ell^1$-norm}. 
\end{defi}

Let $M$ be an oriented closed connected $n$-manifold with fundamental group~$\Gamma$
and let $\alpha$ be a standard $\Gamma$-space. Then the \emph{$\alpha$-parametrised
simplicial volume of~$M$} is defined as
\[ \ifsv M ^\alpha := \| M \|_{L^\infty(\alpha;\Z)} \in \R_{\geq 0}.
\]
Taking the infimum over the set of all isomorphism classes of standard
$\Gamma$-spaces leads to the \emph{integral foliated simplicial
  volume~$\ifsv M$ of~$M$}. Integral foliated simplicial volume fits into
the following chain of inequalities~\cite{mschmidt,loehpagliantini}:
\[ \sv M \leq \ifsv M \leq \stisv M. 
\]
Here,\label{def:stisv}
\[ \stisv M := \inf_{(p\colon N \rightarrow M) \in F(M)} \frac{\sv N_\Z}{|\deg p|} \]
denotes the \emph{stable integral simplicial volume of~$M$} and 
$F(M)$ is the set of all isomorphism classes of finite
connected coverings of~$M$.

%%%%%%%%%%%%%%%%%%%%%%%%%%%%%%%%%%%%%%%%%%%%%%%%%%%%%%%%
\section{The uniform boundary condition}\label{sec:UBC}

The uniform boundary condition in a normed chain complex asks for a uniform
linear bound on fillings of null-homologous cycles. 

\begin{defi}[uniform boundary condition; UBC~\cite{mm}]\label{def:ubc}
	Let~$C_*$ be a normed chain complex and let~$n \in \N$.
	We say that~$C_*$ \emph{satisfies the uniform boundary condition in degree~$n$}, 
	or in short \emph{\UBC n}, if there exists a constant~$K \in \R_{>0}$
	such that:
	For every null-homologous cycle~$c \in C_n$ there exists a filling 
	chain~$b\in C_{n+1}$ with
	\[ \partial b = c \quad \text{and} \quad |b| \leq K \cdot |c|.
	\]
\end{defi}

We briefly recall the results of Matsumoto and Morita~\cite{mm}:

\begin{thm}[UBC {\cite[Theorem~2.8]{mm}}]\label{thm:ubcmm} 
	Let~$M$ be a topological space and~$n \in \N$.
	Then the following are equivalent:
	\begin{enumerate}
		\item The normed chain complex~$C_*(M;\R)$ satisfies~\UBC n.
		\item The homomorphism~$H_b^{n+1}(M;\R) \longrightarrow H^{n+1} (M;\R)$ 
			induced by the inclusion~$C^{n+1}_b(M;\R) \longrightarrow C^{n+1}(M;\R)$ is injective.
	\end{enumerate}
\end{thm}
Here, $C^{*}_b(M;\R) \subset C^{*}(M;\R)$ denotes the subcomplex of the 
singular co\-chain complex of~$M$ with real coefficients 
consisting of bounded linear maps and~$H_b^{*}(M;\R)$
denotes the cohomology of~$C_b^*(M;\R)$, the so-called \emph{bounded cohomology
of~$M$ (with coefficients in~$\R$)}.

Since bounded cohomology of topological spaces with amenable
fundamental group is trivial~\cite{brooks,vbc},
Theorem~\ref{thm:ubcmm} implies the following 
result of which Proposition~\ref{mainprop:Q} is a special case:

\begin{cor}[UBC and amenability]
	Let~$M$ be a topological space with amen\-able fundamental group and let~$n\in \N$.
	Then the chain complex $C_* (M;\R)$ satisfies \UBC n.
\end{cor}

%%%%%%%%%%%%%%%%%%%%%%%%%%%%%%%%%%%%%%%%%%%%%%%%%%%%%%%%
\section{The filling lemma and the lifting lemma}\label{sec:keylemmas}

The F\o lner filling strategy starts with a lifting step and a filling step.
These steps mainly rely on the following Lemmas~\ref{lem:lifting} and~\ref{lem:filling}.

%%%%%%%
\subsection{The filling lemma}

In contractible spaces, boundaries can be filled efficiently. 

\begin{lem}[filling lemma]\label{lem:filling}
  Let $M$ be a contractible topological space, let $A$ be a normed
  $\Z$-module, and let $n \in \N$. For every~$c \in C_n(M;A)$
  there exists a chain~$c' \in C_n(M;A)$ with
  \[ \partial c' = \partial c
     \quad\text{and}\quad
     |c'|_{1} \leq |\partial c|_{1}.
  \]
\end{lem}

Statements of this form can be proved by adapting the original proof
of Frigerio, L\"oh, Pagliantini, and Sauer~\cite[Lemma~6.3]{flps} from
the case of $\Z$-coefficients to general normed coefficients. We will slightly
modify the filling construction in order to improve the filling bound to~$1$.
  
\begin{proof}
  We will construct a cone-type chain contraction
	\[ h \colon C_* (M;A) \longrightarrow C_{*+1}(M;A)
	\]
	of norm~$\leq 1$, inductively on the dimension of the simplices:
	
	Let~$x_0 \in M$. For each $0$-simplex $\sigma$ in~$M$
	we choose a path $h(\sigma) \colon \Delta^1 \longrightarrow M$ from~$x_0$ to~$\sigma (1)$.
	We proceed inductively as follows: 
	Let~$k\in \N$ and let~$\sigma$ be a $k$-simplex in~$M$.
	Consider the map $\partial \Delta^{k+1} \longrightarrow M$ that is given by~$\sigma$ 
	on the 0-th face
	and~$h(\sigma \circ i_{l-1}^k)$ on the~$l$-th face for all~$l \in \{1,\dots, k+1\}$.
	Here, $i_{l-1}^k$ denotes the inclusion of the $(l-1)$-face in the standard $k$-simplex.
	Since~$M$ is contractible, this extends to a map $h(\sigma) \colon \Delta^{k+1} \longrightarrow M$.
	Finally, for all~$k\in \N$ and all~$c = \sum_{j=1}^m a_j \otimes \sigma_j \in C_k (M;A)$ we define
	\[ h(c) := \sum_{j=1}^m a_j \otimes h(\sigma_j) \in C_{k+1} (M;A)
	\]
	By construction,~$\|h\|\leq 1$ and
	for all~$k\in \N$ and all $k$-simplices~$\sigma$ we have
	\begin{align*}
		\partial h ( \sigma ) = \sum_{j=0}^{k+1} (-1)^j \cdot h ( \sigma ) \circ i_j^{k+1}
		            = \sigma + \sum_{j=1}^{k+1} (-1)^j \cdot h ( \sigma \circ i_{j-1}^{k} )
								= \sigma + h ( - \partial \sigma ),
	\end{align*}
	and therefore, we have~$\partial \circ h +h \circ \partial = \id$.
	
	Let~$c \in C_n(M;A)$. We set~$c' := h ( \partial c ) \in C_n(M;A)$. Then we have 
	\[|c'|_{1} \leq ||h|| \cdot |\partial c|_{1} \leq |\partial c|_{1}
	\] and
	\[ \partial c' = \partial h (\partial c) = \partial c - h (\partial \partial c ) = \partial c + 0. \qedhere
	\]
\end{proof}

%%%%%%%
\subsection{The lifting lemma}

Lifting cycles to the universal covering in general will not lead to cycles; however,
the size of the boundary of finite translates of such lifts will basically only grow
like the boundary of the finite set of group elements. 

\begin{lem}[lifting lemma]\label{lem:lifting}
  Let $M$ be a path-connected and locally path-con\-nected space 
	that admits a universal covering,
  let $\Gamma :=\pi_1(M)$, let $A$ be a normed $\Z\Gamma$-module, and
  let $\ucov a \in A \otimes_\Z C_n(\ucov M;\Z)$ be a lift of~$0 \in A
  \otimes_{\Z \Gamma} C_n(\ucov M;\Z)$. Then there is a constant~$C
  \in \R_{>0}$ and a finite set~$S \subset \Gamma$ such that the
  following holds: For every finite subset~$F \subset \Gamma$ we have
  \[ |F \cdot \widetilde a|_{1} \leq C \cdot |\partial_S F|,
  \]
  where $|F \cdot \widetilde a|_1$ is the $\ell^1$-norm on~$A \otimes_\Z C_n(\widetilde M;\Z)$
	induced by the norm on~$A$.
\end{lem}

\begin{notation}
	Here, $\partial_S F = \{\gamma \in F~|~\exists_{s\in S}~\gamma \cdot s \notin F\}$
	denotes the $S$-boundary of~$F$ in~$\Gamma$ and we set~$F \cdot \widetilde a := \sum_{\gamma \in F} \gamma \cdot \widetilde a$,
	where the $\Gamma$-action on~$A \underset{\Z}{\otimes} C_n(\ucov M; \Z)$ is given by
	\[ \gamma \cdot (a \otimes \sigma) := a \cdot \gamma^{-1} \otimes \gamma \cdot \sigma
	\]
	for all~$\gamma \in \Gamma$, $a \in A$ and~$\sigma \in \map(\Delta^n,\ucov M)$.
	The canonical projection
	\[ A \underset{\Z}{\otimes} C_n(\ucov M; \Z) \longrightarrow A \underset{\Z \Gamma}{\otimes} C_n(\ucov M; \Z).
	\]
	justifies the term "lift" in the lifting lemma.
\end{notation}

\begin{proof}
  In this case, the generalisation from the case of
  $\Z$-coefficients~\cite[Lemma~6.2]{flps} is slightly
  more involved (but the proof in spirit is the same):

	Let~$\widetilde{a}=\sum_{j=1}^k{f_{\tau_j}\otimes \tau_j}$ be in reduced form
	in~$A \otimes_\Z C_n(\ucov M;\Z)$,
	i.e., $\tau_i \neq \tau_j$ if $i\neq j$.
	Let
	\[ K:=\big\{C_n(\pi;\Z)(\tau_j)~\big|~j\in\{1,\dots,k\}\big\},
	\]
	where~$\pi \colon \ucov M\longrightarrow M$ is the universal covering.
	For every $\tau\in K$ we choose a lift~$\widetilde{\tau}\in C_n(\ucov M;\Z)$
	that occurs in~$\widetilde{a}$ and we set
	\[ S(\tau):=\{\gamma\in\Gamma~|~\exists_{j\in\{1,\dots,k\}}~\tau_j = \gamma \cdot \widetilde{\tau}\}.
	\]
	Let~$S := \bigcup_{\tau \in K} (S(\tau) \cup S(\tau)^{-1}) \subset \Gamma$.
	Using~$S$, we can write
	\begin{align*}
		\widetilde{a}&=\sum_{\tau \in K}\sum_{s \in S(\tau)}
		          {f_{s \cdot \widetilde{\tau}} \otimes s \cdot \widetilde{\tau}}
                          %\\&
		 =\sum_{\tau \in K}\sum_{s\in S}
		  {f_{s \cdot \widetilde{\tau}} \otimes s \cdot \widetilde{\tau}},
	\end{align*}
	where we set~$f_{s \cdot \widetilde{\tau}}:=0$ for all $s \in S \setminus S(\tau)$. Because 
	$\widetilde{a}$ is a lift of zero, for every~$\tau \in K$, we have
	\begin{align}
	\label{eq:liftofzero}
		\sum_{s \in S}{f_{s \cdot \widetilde{\tau}} \cdot s}=0.
	\end{align}
	
	Let~$F \subset \Gamma$ be a finite subset. The goal is to estimate the $\ell^1$-norm of 
	\begin{align}
	\label{eq:translates}
				F \cdot \widetilde{a}
        = \sum_{\gamma \in F} \gamma \cdot \widetilde a 
        = \sum_{\gamma \in F}\sum_{\tau \in K}\sum_{s \in S}
	  {f_{s \cdot \widetilde{\tau}} \cdot \gamma^{-1} \otimes \gamma \cdot s \cdot \widetilde{\tau}}.
	\end{align}
        To this end, we split~$F \cdot \widetilde a$ in the following
        way as a sum~$X + Y$: We set
	\[ X:=\sum_{\gamma \in F \setminus \partial_S F}\sum_{\tau \in K}\sum_{s \in S}
	   {f_{s \cdot \widetilde{\tau}} \cdot s \cdot \gamma^{-1} \otimes \gamma \cdot \widetilde{\tau}}
           \in A \otimes_\Z C_n(\ucov M;\Z).
	\]
        By definition of the boundary~$\partial_S F$,    
        if $s \in S$ and $\gamma \in F\setminus \partial_S F$, then
        $\gamma \cdot s^{-1} \in F$ (because $S$ is symmetric). 
	Each summand in~$X$ occurs as a summand in~$F \cdot \widetilde{a}$
	(see Equation~(\ref{eq:translates}) with~$\gamma$ set to~$\gamma \cdot s^{-1}$) and the pairs in
	\[ \bigl\{(s \cdot \gamma^{-1},\gamma \cdot \widetilde{\tau}) \bigm|
	  s \in S,\ \gamma \in F \setminus \partial_S F,\ \tau \in K \bigr\}
	\]
	are pairwise distinct. 
	Thus, we can write~$F \cdot \widetilde{a}=X+Y$,
	where~$Y$ is the sum of summands that occur in~$F \cdot \widetilde{a}$ but not in~$X$. 
	We will now estimate~$X$ and~$Y$ separately: 

        We start with~$|Y|_1$: 
        Let $s \in S$ and $\tau \in K$. Then 
	$f_{s \cdot \widetilde{\tau}}$ occurs exactly~$|F|$
	times in~$F \cdot \widetilde{a}$ and exactly~$|F \setminus \partial_S F|$
	times in~$X$. Therefore, it follows that
	each~$f_{s \cdot \widetilde{\tau}}$ occurs exactly~$|\partial_S F|$
	times in~$Y$, which implies 
	\[ |Y|_{1} \leq |\partial_S F|\cdot |K| \cdot |S| \cdot m,
	\]
	where~$m:=\max\bigl\{|f_{s \cdot \widetilde{\tau}}|_A \bigm| s \in S, \tau \in K\bigr\}$.
	
	Finally, the sum~$X$ is zero, because for each~$\tau \in K$
	and~$\gamma \in F \setminus \partial_S F$ we have
	\begin{align*}
		\sum_{s \in S}{f_{s \cdot \widetilde{\tau}}\cdot s \cdot \gamma^{-1}
		 \otimes \gamma \cdot \widetilde{\tau}}
		&=\biggl(\sum_{s \in S} f_{s \cdot \widetilde{\tau}}\cdot s\biggr) \cdot \gamma^{-1}
		\otimes \gamma \cdot \widetilde{\tau}
                %\\
		%&
                \stackrel{(\ref{eq:liftofzero})}{=}0.
	\end{align*}

	We conclude $|F \cdot \widetilde{a}|_{1} = |Y|_{1} \leq C \cdot |\partial_S F|$,
	where $C:=|K| \cdot |S| \cdot m$ depends only on~$\widetilde a$, but not on~$F$.
\end{proof}

%%%%%%%%%%%%%%%%%%%%%%%%%%%%%%%%%%%%%%%%%%%%%%%%%%%%%%%%
\section{Rational UBC for amenable groups}\label{sec:ubcq}

In this section, we will prove Proposition~\ref{mainprop:Q}.

%%%%%%%%%
\subsection{Amenable groups and F\o lner sequences}
We briefly recall the definition of amenable groups 
via left-invariant means
and the characterisation of amenable groups via the F\o lner criterion.
For the details we refer to the literature~\cite{paterson}.

\begin{defi}[amenable]
	A group~$\Gamma$ is called \emph{amenable} if it admits a left-invariant mean, i.e., an $\R$-linear map~$m \colon \ell^\infty (\Gamma,\R) \longrightarrow \R$ that is normalised, positive and left-invariant with respect to the $\Gamma$-action on~$\ell^\infty (\Gamma,\R)$ induced by the right-translation of~$\Gamma$ on itself.
\end{defi}

\begin{thm}[F\o lner sequences]
	Let~$\Gamma$ be a finitely generated group with finite generating set~$S$. Then the following are equivalent:
	\begin{enumerate}
		\item The group~$\Gamma$ is amenable.
		\item The group~$\Gamma$ admits a \emph{F\o lner sequence (with respect to~$S$)}, i.e., a sequence~$(F_k)_{k\in \N}$ of non-empty finite subsets of~$\Gamma$ satisfying 
			\[ \lim_{k\rightarrow \infty} \frac{|\partial_S F_k|}{|F_k|} =0.
			\]
	\end{enumerate}
\end{thm}

Recall that finite groups and abelian groups are amenable
and that the class of amenable groups is closed under taking subgroups, quotient groups and extensions.
Groups that contain a free group of rank~$2$ as subgroup are \emph{not} amenable.

%%%%%%%%%
\subsection{Proof of UBC via F\o lner sets}

Now we are prepared to prove Proposition~\ref{mainprop:Q} via the strategy outlined in the Introduction. The basic steps are also illustrated in Figure~\ref{fig:filling}.

\begin{proof}[Proof of Proposition~\ref{mainprop:Q}]
	Let~$c\in C_n(M;\Q)$ be a null-homologous cycle, i.e., 
	there exists~$b\in C_{n+1}(M;\Q)$ with~$\partial b = c$.
	We will use~$b$ to find a more efficient filling of~$c$.
	Let~$\pi \colon \ucov M \longrightarrow M$ be the universal covering of~$M$
	and let $\Gamma := \pi_1(M)$ be the fundamental group of~$M$.
	
	\emph{Lifting step:} Let~$\ucov c \in C_n(\ucov M;\Q)$ be a $\pi$-lift of~$c$
	with $|\ucov c|_1 \leq |c|_1$ (e.g., by lifting~$c$ simplex by simplex)
	and let $\ucov b \in C_{n+1}(\ucov M;\Q)$ be a $\pi$-lift of~$b$.
	We consider
	\[ \ucov a := \partial \ucov b -\ucov c \in C_n(\ucov M;\Q).
	\]
	Then~$\ucov a$ is a $\pi$-lift of $0 \in C_n(M;\Q)$.
	By the lifting lemma (Lemma~\ref{lem:lifting}) there exist~$C \in \R_{>0}$
	and a finite subset~$S \subset \Gamma$ such that the following holds:
	for all finite sets~$F \subset \Gamma$ we have
	\[ |F \cdot \ucov a|_1 \leq C \cdot |\partial_S F|.
	\]
	The group~$\Lambda := \langle S \rangle_\Gamma \subset \Gamma$ is amenable and finitely generated.
	Let~$(F_k)_{k\in\N}$ be a F\o lner sequence of~$\Lambda$ with respect to~$S$; in particular,
	\[ \lim_{k\rightarrow \infty} \frac{|\partial_S F_k|}{|F_k|} = 0.
	\]
	Then, for all~$k\in \N$ we have
	\begin{align*}
		|\partial (F_k \cdot \ucov b)|_1 
			= |F_k \cdot \ucov a + F_k \cdot \ucov c|_1 
			\leq C \cdot |\partial_S F_k| + |F_k| \cdot |\ucov c|_1.
	\end{align*}

        \begin{figure}
          \begin{center}
            \def\colorb{blue}
            \def\colorc{red}
            \def\chainb#1{%
              \begin{scope}[shift={#1}]
                \fill[\colorb,opacity={0.1}] (0,0) rectangle +(2,2);
                \begin{scope}[\colorb]
                  \draw (0,0) -- (2,2);
                  \draw (2,0) -- (0,2);
                  \draw (0,1) -- (2,1);
                  \draw (1,0) -- (1,2);
                \end{scope}
                \draw[\colorc,very thick] (0,0) rectangle +(2,2);
              \end{scope}
            }
            \begin{tikzpicture}[thick,x=0.5cm,y=0.5cm]
              % lifting
              \chainb{(0,0)}
              \draw[\colorb] (1,-1) node {$\widetilde b$};
              \draw[\colorc] (-1,1) node {$\widetilde c$};
              % translating
              \begin{scope}[shift={(8,0)}]
                \begin{scope}[opacity=0.2]
                \foreach \j in {-1,0,1} {%
                  \foreach \k in {-1,0,1} {%
                    \chainb{(2*\j,2*\k)}
                  }
                }
                \end{scope}
                %\chainb{(0,0)}
                \draw[\colorb] (1,-3) node {$F_k \cdot \widetilde b$};
                \draw[\colorc] (-3.5,1) node {$F_k \cdot \widetilde c$};
                \draw[\colorc,very thick] (-2,-2) rectangle +(6,6);
              \end{scope}
              % filling
              \begin{scope}[shift={(18,0)}]
                \begin{scope}
                  \def\colorb{black!10}
                  \def\colorc{black!10}
                \foreach \j in {-1,0,1} {%
                  \foreach \k in {-1,0,1} {%
                    \chainb{(2*\j,2*\k)}
                  }
                }
                \end{scope}
                \begin{scope}[\colorb]
                  \fill[opacity=0.1] (-2,-2) rectangle +(6,6);
                  \foreach \j in {-2,...,4} {%
                    \draw (1,1) -- (\j,4);
                    \draw (1,1) -- (4,\j);
                    \draw (1,1) -- (\j,-2);
                    \draw (1,1) -- (-2,\j);
                  }
                \end{scope}
                \draw[\colorb] (1,-3) node {$\widetilde b_k$};
                \draw[\colorc] (-3.5,1) node {$F_k \cdot \widetilde c$};
                \draw[\colorc,very thick] (-2,-2) rectangle +(6,6);
              \end{scope}
            \end{tikzpicture}
          \end{center}

          \caption{Lifting, translating, and filling chains, schematically}
          \label{fig:filling}
        \end{figure}
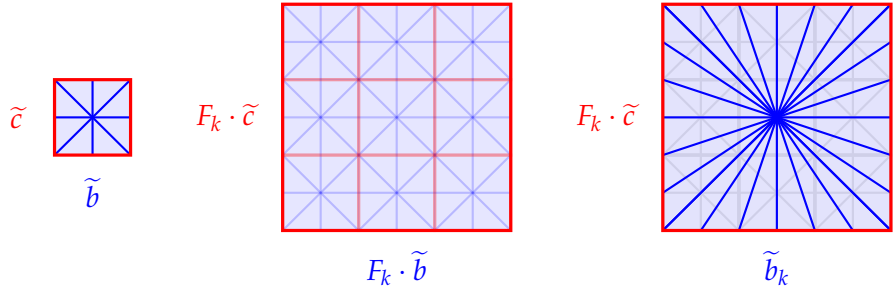
        
	\emph{Filling step:} By the filling lemma (Lemma~\ref{lem:filling}), for every~$k \in \N$
	there exists a chain~$\ucov b_k \in C_{n+1}(\ucov M;\Q)$ with~$\partial \ucov b_k = \partial (F_k \cdot \ucov b)$ and
	\[ |\ucov b_k|_{1} \leq |\partial (F_k \cdot \ucov b)|_{1} 
		\leq C \cdot |\partial_S F_k| + |F_k| \cdot |\ucov c|_1.
	\]
	
	\emph{Quotient step:} For all~$k \in \N$ we define
	\[ b_k := \frac{1}{|F_k|} \cdot C_{n+1} (\pi;\Q) (\ucov b_k) \in C_{n+1} (M;\Q).
	\]
	By construction, we have
	\begin{align*}
		\partial b_k = \frac{1}{|F_k|} \cdot C_{n+1}(\pi;\Q) \bigl(\partial (F_k \cdot \ucov b)\bigr)
		             = \frac{1}{|F_k|} \cdot \partial (|F_k| \cdot  b)
								 =\partial b = c
	\end{align*}
	and
	\begin{align*}
		|b_k|_1 \leq \frac{1}{|F_k|} \cdot |\ucov b_k|_1
		        \leq C \cdot \frac{|\partial_S F_k|}{|F_k|} + |\ucov c|_1 
						\leq C \cdot \frac{|\partial_S F_k|}{|F_k|} + |c|_1.
	\end{align*}
	Because~$(F_k)_{k \in \N}$ is a F\o lner sequence, for every~$\epsilon\in \R_{>0}$ 
	there exists~$k \in \N$ such that $\partial b_k =c$ and
	\[ |b_k|_1 \leq (1+\epsilon) \cdot |c|_1,
	\]
	which is slightly stronger than the statement of Proposition~\ref{mainprop:Q}.
\end{proof}

%%%%%%%%%%%%%%%%%%%%%%%%%%%%%%%%%%%%%%%%%%%%%%%%%%%%%%%%
\section{Stable integral UBC for amenable groups}\label{sec:stiubc}

Taking improved F\o lner sequences allows to prove the uniform boundary
condition for the stable integral $\ell^1$-norm (Theorem~\ref{mainthm:stisv}). 

%%%%%
\subsection{Improved F\o lner sequences}
We use the following result of Deninger and Schmidt~\cite[Proposition~5.5]{deningerschmidt}, which is a 
reformulation of the Rokhlin Lemma for amenable residually finite groups of Weiss~\cite[Theorem~1]{weiss}.

\begin{prop}[improved F\o lner sequences]
\label{prop:rokhlinresfin}
	Let~$\Gamma$ be a countable amenable residually finite group
	and let~$S \subset \Gamma$ be a finite subset.
	Then there exists a sequence~$(\Gamma_k)_{k\in \N}$ of 
	decreasing finite index normal subgroups of~$\Gamma$
	with~$\bigcap_{k \in \N} \Gamma_k = \{e\}$
	and a F\o lner sequence~$(F_k)_{k\in\N}$ of~$\Gamma$ 
	with respect to~$S$ such that:
	For all~$k \in \N$ the F\o lner set~$F_k$ is a set 
	of representatives for~$\Gamma/\Gamma_k$.
\end{prop}

%%%%%
\subsection{Proof of Theorem~\ref{mainthm:stisv}}
	Let~$c\in C_n(M;\Z)$ be a null-homologous cycle, i.e., 
	there exists~$b\in C_{n+1}(M;\Z)$ with~$\partial b = c$.
	Let~$\pi \colon \ucov M \longrightarrow M$ be the universal covering of~$M$
	and~$\Gamma := \pi_1(M)$ the fundamental group of~$M$.
	
	The lifting step is analogous to the lifting
	step in the proof of Proposition~\ref{mainprop:Q}.
	By Proposition~\ref{prop:rokhlinresfin}
	there exists a sequence~$(\Gamma_k)_{k\in \N}$ of decreasing 
	finite index normal subgroups of~$\Gamma$
	with~$\bigcap_{k \in \N} \Gamma_k = \{e\}$
	and a F\o lner sequence~$(F_k)_{k\in\N}$ of~$\Gamma$ with 
	respect to~$S$ (from the lifting step) such that:
	For all~$k \in \N$ the F\o lner set~$F_k$ is a set of 
	representatives for~$\Gamma/\Gamma_k$.
	We construct~$\ucov b_k \in C_{n+1} (\ucov M;\Z)$ for all~$k\in \N$ as 
	in the filling step in the proof of Proposition~\ref{mainprop:Q}.
	
	\emph{Quotient step:} Let~$k \in \N$. We write~$p_k\colon M_k \longrightarrow M$ 
	for the covering of~$M$ associated to~$\Gamma_k < \Gamma$ (of degree~$|F_k|$)
	and~$\pi_k \colon \ucov M \longrightarrow M_k$ for the universal covering of~$M_k$.
	Then, we define
	\[ b_k := C_{n+1} (\pi_k;\Z) (\ucov b_k) \in C_{n+1} (M_k;\Z).
	\]
	In~$C_{n}(M_k;\Z)$ we have
	\begin{align*}
		\partial b_k = C_{n+1}(\pi_k;\Z) \bigl(\partial (F_k \cdot \ucov b)\bigr)
		             = C_{n+1}(\pi_k;\Z) (F_k \cdot \partial \ucov b) =: c_k
	\end{align*}
	and by construction~$c_k$ is the \emph{full~$p_k$-lift of~$c$}, i.e.,
	if $c = \sum_{j=1}^m a_j \cdot \sigma_j$ then
	\begin{align}
	\label{eq:fulllift}
		c_k = \sum_{j=1}^m a_j \cdot \biggl(
						\sum_{\substack{\tau \in \map(\Delta^n,M_k)\\p_k\circ \tau = \sigma_j}} \tau\biggr).
	\end{align}
	We estimate (where $C \in \R_{>0}$ is the constant found in
        the lifting step)
	\begin{align*}
		|b_k|_1 \leq |\ucov b_k|_1
		        \leq C \cdot |\partial_S F_k| + |F_k| \cdot |\ucov c|_1 
						\leq |F_k| \cdot \Bigl( C \cdot \frac{|\partial_S F_k|}{|F_k|} + |c|_1 \Bigr).
	\end{align*}
	Because~$(F_k)_{k \in \N}$ is a F\o lner sequence, for every~$\epsilon\in \R_{>0}$ 
	there exists~$k \in \N$ such that $\partial b_k =c_k$ and
	\[ |b_k|_1 \leq |F_k| \cdot (1 + \epsilon) \cdot |c|_1,
	\]
	which is slightly stronger than the statement of Theorem~\ref{mainthm:stisv}. \qed

%%%%%%%%%%%%%%%%%%%%%%%%%%%%%%%%%%%%%%%%%%%%%%%%%%%%%%%%
\section{Parametrised UBC for tori}\label{sec:parsvab}

In order to prove the uniform boundary condition
for the parametrised $\ell^1$-norm, we will perform the division
by the order of the F\o lner sets on the level of the measure space.
This is done through suitable versions of the Rokhlin lemma.

%%%%
\subsection{Rokhlin lemma for free abelian groups}

In the abelian case, we will use the following version of the Roklin
lemma~\cite[Theorem~3.1]{conze}.

\begin{thm}[Rokhlin lemma for free abelian groups]\label{thm:rokhlinab}
	Let~$d\in \N_{>0}$ and let~$(X,\mu)$ be an (essentially) free standard $\Z^d$-space.
	Then for every~$k\in \N$ and every~$\epsilon \in \R_{> 0}$
	there exists a measurable subset~$A \subset X$ such that the sets
	\[ (\gamma \cdot A )_{\gamma \in F_k}
	\]
	with~$F_k := \{ 0, \dots, k \}^d \subset \Z^d$ are pairwise disjoint and
	\[ \mu ( X \setminus F_k \cdot A) < \epsilon.
	\]
\end{thm}

%%%%
\subsection{Proof of Theorem~\ref{mainthm:parsvab}}
	Let~$c\in C_n(M;\alpha)$ be a null-homologous cycle, i.e., 
	there exists~$b\in C_{n+1}(M;\alpha)$ with~$\partial b = c$.
	We will use~$b$ to find a more efficient filling of~$c$.
	Let
	\[ p \colon L^\infty(X;\Z)\underset{\Z}{\otimes} C_{*}(\widetilde{M};\Z) \longrightarrow L^\infty(X;\Z)\underset{\Z\Gamma}{\otimes} C_{*}(\widetilde{M};\Z)
	\]
	be the canonical projection.
	
	\emph{Lifting step:} Let~$\ucov c$ be a $p$-lift of~$c$
	with $|\ucov c|_1 \leq |c|_1$ (e.g., by lifting~$c$ simplex by simplex)
	and let $\ucov b $ be a $p$-lift of~$b$.
	We now consider
	\[ \ucov a := \partial \ucov b -\ucov c \in L^\infty(X;\Z)\underset{\Z}{\otimes} C_{n}(\widetilde{M};\Z).
	\]
	Then~$\ucov a$ is a $p$-lift of $0 \in C_n(M;\alpha)$.
	By the lifting lemma (Lemma~\ref{lem:lifting}), there exist~$C \in \R_{>0}$
	and a finite subset~$S \subset \Gamma$ such that the following holds:
	for all finite sets~$F \subset \Gamma$ we have
	\[ |F \cdot \ucov a|_1 \leq C \cdot |\partial_S F|.
	\]
        For~$k \in \N$ we define
        \[ F_k := \{0,\dots, k\}^d \subset \Gamma
        \]
        via an isomorphism~$\Gamma \cong \Z^d$. 
	Then~$(F_k^{-1})_{k \in \N}$ is a F\o lner sequence for~$\Gamma$ with respect to~$S$ in the sense that 
	\[ \lim_{k\rightarrow \infty} \frac{|\partial_S (F_k^{-1})|}{|F_k|} = 0.
	\]
	Then, for all~$k\in \N$ we have
	\begin{align*}
		|\partial (F_k^{-1} \cdot \ucov b)|_1 &= |F_k^{-1} \cdot \ucov a + F_k^{-1} \cdot \ucov c|_1  \leq C \cdot |\partial_S F_k^{-1}| + |F_k| \cdot |\ucov c|_1.
	\end{align*}
	
	Let~$k \in \N$ and~$\epsilon \in \R_{>0}$. By the Rokhlin lemma for 
	free abelian groups (Theorem~\ref{thm:rokhlinab}), 
	there exists a measurable subset~$A_k \subset X$ such that the sets
	$(\gamma \cdot A_k )_{\gamma \in F_k}
	$
	are pairwise disjoint and the complement
	\[ B_k := X \setminus F_k \cdot A_k
	\]
	has measure less than~$\epsilon$.
	
	\emph{Quotient step:} Since~$L^\infty (X; \Z)$ is a $L^\infty(X;\Z)$-$\Z$-bimodule,  
	%(but in general no $L^\infty(X;\Z)$-$\Z \Gamma$-bimodule)
        it follows 
	that
	\[ L^\infty (X; \Z) \underset{\Z}{\otimes} C_n(\ucov M;\Z)
	\]
	is a left-$L^\infty (X; \Z)$-module. Therefore, we can define
	\[ \ucov b_{k,\gamma} := \chi_{\gamma \cdot A_k} 
	      \cdot (F_k^{-1} \cdot \ucov b)
				\in L^\infty (X;\Z) \underset{\Z}{\otimes} C_n(\ucov M;\Z)
 	\]
	for all~$\gamma \in F_k$. Then we have
	\begin{align*}
	  F_k^{-1} \cdot \ucov b &= \chi_X \cdot (F_k^{-1} \cdot \ucov b)
          %\\
	  %&
          = \sum_{\gamma \in F_k} \chi_{\gamma \cdot A_k} 
											   \cdot (F_k^{-1} \cdot \ucov b)
												+ \chi_{B_k} \cdot (F_k^{-1} \cdot \ucov b)\\
											&= \sum_{\gamma \in F_k} \ucov b_{k,\gamma} +  \chi_{B_k} \cdot (F_k^{-1} \cdot \ucov b).
	\end{align*}
	Because the chains~$(\partial \ucov b_{k,\gamma})_{\gamma \in F_k}$ have pairwise disjoint support, we obtain 
	\[ \biggl| \sum_{\gamma \in F_k} \partial \ucov b_{k,\gamma} \biggr|_1 = \sum_{\gamma \in F_k} |\partial \ucov b_{k,\gamma}|_1
	\]
	and by the box principle it follows that there exists~$\gamma_0 \in F_k$ with
        \[ |\partial \widetilde b_{k,\gamma_0}|_1
           \leq
           \frac1{|F_k|} \cdot \bigl( |F_k^{-1} \cdot \partial \widetilde b|_1 + | \chi_{B_k} \cdot (F_k^{-1} \cdot \partial \widetilde b)|_1 \bigr).
        \]
        We write~$\ucov b = \sum_{j=1}^m a_j \otimes \sigma_j \in
        L^\infty (X;\Z) \underset{\Z}{\otimes} C_{n+1}(\ucov M;\Z)$ in
        reduced form and set
        \[ |\widetilde b|_{1,\infty} := \sum_{j=1}^m |a_j|_\infty.
        \]
        Then
	\begin{align*}
		|\partial \ucov b_{k,\gamma_0}|_1 &\leq \frac{1}{|F_k|} \cdot \bigl( | F_k^{-1} \cdot \partial \ucov b|_1 + \mu(B_k) \cdot |F_k| \cdot (n+2) \cdot |\ucov b|_{1,\infty} \bigr)\\
		   &\leq C \cdot \frac{|\partial_S (F_k^{-1})|}{|F_k|} + |c|_1 + \epsilon \cdot (n+2) \cdot |\ucov b|_{1,\infty}.
	\end{align*}
	
	\emph{Filling step:} By the filling lemma (Lemma~\ref{lem:filling}),
	there exists a parametrised chain~$\ucov b_{k,\gamma_0}' \in L^\infty (X;\Z) \underset{\Z}{\otimes} C_{n+1}(\ucov M;\Z)$ 
	with~$\partial \ucov b_{k,\gamma_0}' =  \partial \ucov b_{k,\gamma_0}$ and
	\[ |\ucov b_{k,\gamma_0}'|_{1} 
	   \leq |\partial \ucov b_{k,\gamma_0}|_{1}
	   \leq C \cdot\frac{|\partial_S (F_k^{-1})|}{|F_k|} + |c|_1 + \epsilon \cdot (n+2) \cdot |\ucov b|_{1,\infty}.
	\]
	Let
	\[ b_{k,\gamma_0} := p(\ucov b_{k,\gamma_0}) \in C_{n+1} (M;\alpha) 
		\quad \text{and} \quad 
		 b_{k,\gamma_0}' :=  p(\ucov b_{k,\gamma_0}') \in C_{n+1} (M;\alpha). 
	\]
	 
	Because $\Gamma$ is abelian and $\chi_{\gamma\cdot A_k}= \chi_{A_k} \cdot \gamma^{-1}$
	holds for all $\gamma \in \Gamma$, we obtain in~$C_n(M;\alpha)$ that 
	\begin{align*}
		\partial b_{k,\gamma_0}' &= \partial b_{k,\gamma_0}\\
		                &= \sum_{j=1}^m \sum_{\gamma \in F_k} \chi_{\gamma_0 \cdot A_k} \cdot 
											( a_j \cdot \gamma )  \otimes  \gamma^{-1} \cdot \partial \sigma_j\\
										&= \sum_{j=1}^m \sum_{\gamma \in F_k} \chi_{A_k} \cdot 
											( a_j \cdot \gamma \cdot \gamma_0 ) \otimes
											\gamma_0^{-1} \cdot \gamma^{-1} \cdot \partial \sigma_j\\
										&= \sum_{j=1}^m \sum_{\gamma \in F_k} \chi_{A_k} \cdot 
											( a_j \cdot \gamma_0 \cdot \gamma ) \otimes 
											\gamma^{-1} \cdot \gamma_0^{-1} \cdot \partial \sigma_j\\
										&= \sum_{j=1}^m \sum_{\gamma \in F_k} \chi_{\gamma \cdot A_k} \cdot 
											( a_j \cdot \gamma_0 ) \otimes \gamma_0^{-1} \cdot \partial \sigma_j,
	\end{align*}
  which almost looks like~$c$. 
	We define the correction term 
	\begin{align*}
		r_k := \sum_{j=1}^m \chi_{\gamma_0 \cdot B_k} \cdot a_j \otimes \sigma_j 
		    = \sum_{j=1}^m \chi_{B_k} \cdot (a_j \cdot \gamma_0) \otimes \gamma_0^{-1} \cdot \sigma_j
	\end{align*}
	and observe that the following holds in~$C_n(M;\alpha)$:
	\begin{align*}
		\partial ( b_{k,\gamma_0}' + r_k) 
		= \sum_{j=1}^m \chi_{X} \cdot (a_j \cdot \gamma_0) \otimes \gamma_0^{-1} \cdot \partial \sigma_j
		= \sum_{j=1}^m a_j \otimes \partial \sigma_j
		=\partial b
		= c.
	\end{align*}
	Finally, we have
	\begin{align*}
		|b_{k,\gamma_0}'+r_k|_1 &\leq |b_{k,\gamma_0}'|_1 + |r_k|_1\\
		                       &\leq C \cdot \frac{|\partial_S (F_k^{-1})|}{|F_k|} + |c|_1 + \epsilon \cdot (n+2) \cdot |\ucov b|_{1,\infty} + \epsilon \cdot |\ucov b|_{1,\infty}\\
													 &\leq C \cdot \frac{|\partial_S (F_k^{-1})|}{|F_k|} + |c|_1 + \epsilon \cdot (n + 3) \cdot |\ucov b|_{1,\infty}.
	\end{align*}
	Because~$(F_k^{-1})_{k \in \N}$ is a F\o lner sequence for~$\Gamma$ with respect to~$S$, 
	for~$k \rightarrow \infty$ and~$\epsilon \rightarrow 0$ 
	the chains~$b_{k,\gamma_0}'+r_k$ are efficient fillings of~$c$.\qed

%\begin{rem}
% It is not hard to see that all these arguments generalise
%  to the case where the universal covering is not contractible
%  but sufficiently highly connected.
%\end{rem}

%%%%%%%%%%%%%%%%%%%%%%%%%%%%%%%%%%%%%%%%%%%%%%%%%%%%%%%%
\section{Mixed UBC for amenable groups}\label{sec:parsv}

%%%%
\subsection{Ornstein-Weiss}
We need the following modification~\cite[Theorem~5.2]{sauer} of the generalized Rokhlin lemma of Ornstein-Weiss~\cite[Theorem~5]{ow}:

\begin{thm}[Ornstein-Weiss]
	\label{thm:ow}
	Let~$\Gamma$ be a countable amenable group, let~$(X,\mu)$ be an (essentially) free standard $\Gamma$-space.
	Then for every finite subset~$S \subset \Gamma$ and every~$\epsilon \in \R_{>0}$ 
	there exists an~$N \in \N$, finite subsets~$F_1, \dots, F_N \subset \Gamma$, and 
	Borel subsets~$A_1, \dots, A_N \subset X$ such that the following holds:
	\begin{itemize}
	        \item For every~$k\in \{ 1,\dots, N \}$ we have
                  \[ \frac{|\partial_S (F_k^{-1})|}{|F_k|}
                     = \frac{\bigl| \{\gamma \in F_k \bigm| \exists_{s \in S^{-1}}\ s \cdot \gamma \not\in F_k \}\bigr|}{|F_k|}< \epsilon.\]
		\item For every~$k \in \{ 1, \dots, N \}$ the sets~$\gamma \cdot A_k$ with~$\gamma \in F_k$
					are pairwise disjoint.
		\item The sets~$F_k \cdot A_k$ with~$k \in \{ 1, \dots, N \}$ are pairwise disjoint.
		\item The complement~$B := X \setminus \bigcup_{k=1}^N F_k \cdot A_k$ has measure less than~$\epsilon$.
	\end{itemize}
\end{thm}

%%%%
\subsection{Proof of Theorem~\ref{mainthm:parsv}}
	Let~$c\in C_n(M;\Z)$ be a null-homologous cycle, i.e., 
	there exists~$b\in C_{n+1}(M;\Z)$ with~$\partial b = c$.
	We will use~$b$ to find a more efficient filling of~$c$.
	Let~$\pi \colon \ucov M \longrightarrow M$ 
	be the universal covering of~$M$.
	
	\emph{Lifting step:} Let~$\ucov c$ be a $\pi$-lift of~$c$
	with $|\ucov c|_1 \leq |c|_1$ (e.g., by lifting~$c$ simplex by simplex)
	and let $\ucov b $ be a $\pi$-lift of~$b$.
	We now consider
	\[ \ucov a := \partial \ucov b -\ucov c \in C_{n}(\ucov M;\Z).
	\]
	Then~$\ucov a$ is a $\pi$-lift of $0 \in C_n(M;\Z)$.
	By the lifting lemma (Lemma~\ref{lem:lifting}) there exist~$C \in \R_{>0}$
	and a finite subset~$S \subset \Gamma$ such that the following holds:
	For all finite sets~$F \subset \Gamma$ we have
	\[ |F \cdot \ucov a|_1 \leq C \cdot |\partial_S F|.
	\]
	The group~$\Lambda := \langle S \rangle_\Gamma \subset \Gamma$ is amenable and finitely generated.
	
	Let~$\epsilon \in \R_{>0}$. We apply Theorem~\ref{thm:ow} to the (essentially) free standard
	$\Lambda$-space~$\res^\Gamma_\Lambda \alpha = \Lambda \curvearrowright (X,\mu)$: Thus, 
	there exists an~$N \in \N$, finite subsets~$F_1, \dots, F_N \subset \Lambda$ and 
	Borel subsets~$A_1, \dots, A_N, B \subset X$ with the properties listed in Theorem~\ref{thm:ow}.

	\emph{Filling step:} By the filling lemma (Lemma~\ref{lem:filling}),
	for all~$k\in \{1, \dots, N\}$ there exists~$\ucov b_{k} \in C_{n+1}(\ucov M;\Z)$ 
	with~$\partial \ucov b_{k} = \partial (F_k^{-1} \cdot \ucov b) = F_k^{-1} \cdot \partial \ucov b$ and
	\[ |\ucov b_{k}|_{1} 
	   \leq |F_k^{-1} \cdot \ucov c + F_k^{-1} \cdot \ucov a|_{1} 
	   \leq |F_k| \cdot |c|_1 + C \cdot |\partial_S (F_k^{-1})|
		 \leq |F_k| \cdot |c|_1 + C \cdot \epsilon \cdot |F_k|.
	\]
	
	\emph{Quotient step:} We define
	\[ b_\varepsilon := \sum_{k=1}^N \chi_{A_k} \otimes \ucov b_k + \chi_B \otimes \ucov b \in C_{n+1} (M; \alpha).
	\]
	Then, in~$C_n (M;\alpha)$ the following computation holds
	\begin{align*}
		\partial  b_\varepsilon &= \sum_{k=1}^N \chi_{A_k} \otimes \partial \ucov b_k + \chi_B \otimes \partial \ucov b \\
								&= \sum_{k=1}^N \chi_{A_k} \otimes F_k^{-1} \cdot \partial \ucov b + \chi_B \otimes \partial \ucov b \\
								&= \sum_{k=1}^N \sum_{\gamma \in F_k} \chi_{\gamma \cdot A_k} \otimes \partial \ucov b + \chi_B \otimes \partial \ucov b \\
								&= 1 \otimes \partial \ucov b
	\end{align*}
	because $X = B \cup \bigl( \bigcup_{k = 1}^N \bigcup_{\gamma \in F_k} \gamma \cdot A_k\bigr)$
	is a disjoint union.
        Since $\partial \widetilde b$ is a $\pi$-lift of~$c$,
        we obtain $\partial \widetilde b_\varepsilon = c$, if we
        view~$c$ as a chain in~$C_n(M;\Z) \subset C_n(M;\alpha)$
        via the inclusion~$\Z \hookrightarrow L^\infty(\alpha;\Z)$
        as constant functions. 
	Finally, we have
	\begin{align*}
		|b_\varepsilon|_1 &\leq \sum_{k=1}^N \mu(A_k) \cdot |\ucov b_k|_1 + \mu(B) \cdot |\ucov b|_1 \\
						&\leq \sum_{k=1}^N \mu(A_k) \cdot \bigl(|F_k| \cdot |c|_1 + C \cdot \epsilon \cdot |F_k|\bigr) + \mu(B) \cdot |\ucov b|_1 \\
						&= \bigl(|c|_1 + C \cdot \epsilon\bigr) \cdot \sum_{k=1}^N \mu(A_k) \cdot |F_k| + \mu(B) \cdot |\ucov b|_1 \\
						&= \bigl(|c|_1 + C \cdot \epsilon\bigr) \cdot \mu \Bigl( \bigcup_{k=1}^N F_k \cdot A_k \Bigr) + \mu(B) \cdot |\ucov b|_1 \\
						&\leq |c|_1 + C \cdot \epsilon + \epsilon \cdot |\ucov b|_1.
	\end{align*}
        Therefore, 
	for~$\epsilon \rightarrow 0$ the chains~$b_\varepsilon \in C_{n+1} (M;\alpha)$ are efficient fillings of the chain~$c \in C_n (M;\Z) \subset C_n (M;\alpha)$.\qed

%%%%%%%%%%%%%%%%%%%%%%%%%%%%%%%%%%%%%%%%%%%%%%%%%%%%%%%%
\section{Integral UBC}\label{sec:ubcz}

We will now briefly discuss the uniform boundary condition for
the integral singular chain complex.

%%%%%%%%%
\subsection{The integral uniform boundary condition for the circle}

We start with a simple example, namely the circle (and degree~$1$).

\begin{prop}[integral \UBC 1 for the circle]
  The chain complex~$C_*(S^1;\Z)$ satisfies \UBC 1. More precisely:
  If $c \in C_1(S^1;\Z)$ is a null-homologous cycle, then there
  exists a filling chain~$c \in C_2(S^1;\Z)$ satisfying
  \[ \partial b = c
     \quad\text{and}\quad
     |b|_1 \leq 3 \cdot |c|_1.
  \]
\end{prop}

\begin{proof}
  Let $c \in C_1(S^1;\Z)$ be a null-homologous cycle. We use a
  Hurewicz argument to construct an efficient filling
  of~$c$. Therefore, it is convenient to normalise~$c$ as follows:
  \begin{itemize}
  \item Using the fact that $S^1$ is path-connected, we can find
    $b_+ \in C_2(S^1;\Z)$ with $|b_+|_1 \leq 2 \cdot |c|_1$
    such that
    \[ c_+ := c - \partial b \in C_1(S^1;\Z)
    \]
    satisfies~$|c_+|_1 \leq |c|_1$ and such that every singular simplex
    in~$c_+$ maps the boundary of~$\Delta^1$ to the basepoint of~$S^1$~\cite[Chapter~9.5]{tomdieck}.
  \item Splitting the integral coefficients of the chain~$c_+$ into unit steps,
    we write~$c_+ = \sum_{j=1}^m a_j \cdot \sigma_j$ with
    $a_j \in \{-1,1\}$ for all~$j \in \{1,\dots,m\}$ and $|c_+|_1 =
    m$.
  \end{itemize}
  We now use the connection between the fundamental group and~$H_1(S^1;\Z)$:
  Let us consider the based loop
  \[ f := \sigma_1^{a_j} * \dots * \sigma_m^{a_m} \colon S^1 \longrightarrow S^1
  \]
  (using the equidistant partition of~$S^1$ into $m$ segments; without
  loss of generality we may assume~$m \neq 0$).
  Here, if $a_j = -1$, the symbol~$\sigma_j^{a_j}$ denotes the
  reversed loop of~$\sigma_j$.  Because $[c_+] = [c] = 0 \in H_1(S^1;\Z)$,
  we obtain $[f]_* = 0 \in \pi_1(S^1)$ from the Hurewicz theorem. Thus,
  there exists a continuous map~$F \colon D^2 \longrightarrow S^1$ extending~$f$,
  i.e., $F|_{\partial D^2} = f$.

  The filling~$F$ of~$f$ leads to a filling of~$c_+$: Let $j \in
  \{1,\dots,m\}$.  Then $\tau_j \colon \Delta^2 \longrightarrow S^1$
  denotes the restriction of~$F$ to the $j$-th segment of~$D^2$
	(Figure~\ref{fig:ZUBCS1}); if
  $a_j = 1$ we take the positive orientation, if $a_j = -1$ we take
  the negative orientation (see Figure~\ref{fig:ZUBCS1} for the exact
  orientation). We then set
  \[ b_+ := \sum_{j=1}^m a_j \cdot \tau_j \in C_2(S^1;\Z).
  \]
  A straightforward computation shows that
  \[ \partial b_+ = \sum_{j=1}^m a_j \cdot \partial \tau_j
                  = \sum_{j=1}^m a_j \cdot \sigma_j = c_+
  \]
  (because the ``inner'' terms cancel) and 
  $|b_+|_1 \leq m = |c_+|_1 \leq |c|_1.
  $ 
  Combining $b$ and $b_+$ gives the desired filling of~$c$.
  %: We have
  %\[ c = c - c_+ + c_+ = \partial (b + b_+)
  %\quad
  %\text{and}
  %\quad
  %   |b+b_+|_1 \leq |b|_1 + |b_+|_1 \leq 3 \cdot |c|_1,
  %\]
  %as claimed.  
\end{proof}

\begin{figure}
  \begin{center}
    \begin{tikzpicture}[thick]
      \begin{scope}[black!50]
        \draw (0,0) circle (1);
        \draw[->] (1.3,0) arc (0:30:1.3);
        \draw (15:1.5) node {$f$};
        \draw (0,-0.5) node {$F$};
      \end{scope}
      \begin{scope}[blue]
        \fill[blue!10] (0,0) -- (30:1) arc (30:80:1) -- (80:1) -- cycle;
        \draw[->] (0,0) -- (30:1);
        \draw[->] (0,0) -- (80:1);
        \draw (105:0.5) node {$\tau_j$};
        \draw (55:0.2) node {\tiny $0$};
        \draw (37:0.85) node {\tiny $1$};
        \draw (73:0.85) node {\tiny $2$};
      \end{scope}
      \begin{scope}[red]
        \draw[->] (30:1) arc (30:80:1);
        \draw (55:1.3) node {$\sigma_j$};
      \end{scope}
      \draw (0,-2) node {if $a_j = 1$};
      %%%%%%%%%%%%%%%%%%%%%%%%%%%%%%%%%%%%%
      \begin{scope}[shift={(5,0)}]
      \begin{scope}[black!50]
        \draw (0,0) circle (1);
        \draw[->] (1.3,0) arc (0:30:1.3);
        \draw (15:1.5) node {$f$};
        \draw (0,-0.5) node {$F$};
      \end{scope}
      \begin{scope}[blue]
        \fill[blue!10] (0,0) -- (30:1) arc (30:80:1) -- (80:1) -- cycle;
        \draw[->] (0,0) -- (30:1);
        \draw[->] (0,0) -- (80:1);
        \draw (105:0.5) node {$\tau_j$};
        \draw (55:0.2) node {\tiny $0$};
        \draw (37:0.85) node {\tiny $2$};
        \draw (73:0.85) node {\tiny $1$};
      \end{scope}
      \begin{scope}[red]
        \draw[->] (80:1) arc (80:30:1);
        \draw (55:1.3) node {$\sigma_j$};
      \end{scope}
      \draw (0,-2) node {if $a_j = -1$};
      \end{scope}
    \end{tikzpicture}
  \end{center}
  
  \caption{Construction of the filling chain}
  \label{fig:ZUBCS1}
\end{figure}
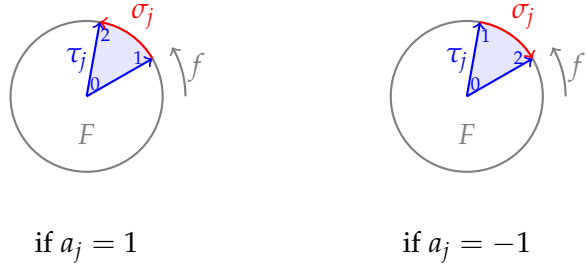

\begin{rem}
  The same Hurewicz argument also can be used to show the following:
  \begin{enumerate}
  \item Let $M$ be a topological space such that the fundamental group
    of every path-connected component is abelian. Then $C_*(M;\Z)$
    satisfies \UBC 1.
  \item Let $n \in \N_{\geq 2}$ and let $M$ be an $(n-1)$-connected topological
    space. Then $C_*(M;\Z)$ satisfies \UBC n.
  \end{enumerate}
\end{rem}

% self-covering arguments?!

%%%%
\subsection{Discussion of integral UBC for general spaces}

However, for more general spaces and degrees, the situation gets
more involved and the general picture is unknown.

\begin{prop}
  Let $n \in \N_{\geq 3}$. Then there exists a simply connected space~$M$
  such that $C_*(M;\Z)$ does \emph{not} satisfy \UBC n.
\end{prop}
\begin{proof}
  We will construct such an example using the following input: Let $N$
  be an oriented closed (connected) $n$-manifold and let $(M_k)_{k \in
    \N}$ be a sequence of oriented compact connected $(n+1)$-manifolds
  such that for every~$k \in \N$ we have
  \[ \partial M_k \cong N
     \quad\text{and}\quad
     b_2(M_k;\Z) \geq k.
  \]
  We then set
  \[ M := \prod_{k \in \N} M_k
  \]
  (if one prefers an example of dimension~$n+1$, one can
  also apply the same arguments to~$M := \bigvee_{k \in \N} M_k$).
  
  We now prove that $C_*(M;\Z)$ does \emph{not} satisfy \UBC n: Let
  $c \in C_n(N;\Z)$ be a fundamental cycle of~$N$.  As a preparation,
  for~$k \in \N$ we consider the fillings of~$c$ in~$C_*(M_k;\Z)$. The
  Betti number estimate~\cite[Example~14.28]{lueckl2}\cite[Lemma~4.1]{flps} 
  for integral simplicial volume generalises to the relative case and
  shows that
  \[ b_j(M_k;\Z) \leq \| M_k,\partial M_k \|_{\Z}
  \]
  holds for all~$j \in \N$. In particular,
  \[ \| M_k, \partial M_k \|_{\Z} \geq b_2(M_k;\Z) \geq k.
  \]
  Let $b \in C_{n+1}(M_k;\Z)$ be a chain with~$\partial b = c$,
  where we view~$c \in C_n(N;\Z)$ as an element of~$C_n(M_k;\Z)$
  via~$\partial M_k \cong N$. Then $b$ is a relative fundamental
  cycle of~$(M_k,\partial M_k)$, and so
  \[ |b|_1 \geq \| M_k,\partial M_k\|_{\Z} \geq k.
  \]
  We now come back to~$M$: For each~$k \in \N$, we choose a basepoint
  in~$M_k$; thus we obtain an inclusion~$i_k \colon M_k
  \longrightarrow M$. If $p_k \colon M \longrightarrow M_k$
  denotes the canonical projection, we have $p_k \circ i_k = \id_{M_k}$.
  For~$k \in \N$ we consider
  \[ c_k := C_n(i_k;\Z)(c) \in C_n(M;\Z).
  \]
  Then $|c_k|_1 \leq |c|_1$ and $c_k$ is null-homologous (because it
  can be filled by any relative fundamental cycle in the factor~$M_k$).
  However, if $b \in C_{n+1}(M;\Z)$ is a filling of~$c_k$, then
  $b_k := C_{n+1}(p_k;\Z)(b) \in C_{n+1}(M_k;\Z)$ is a filling of~$c$
  and thus
  \[ |b|_1 \geq |b_k|_1 \geq k.
  \]
  Therefore, $C_*(M;\Z)$ does \emph{not} satisfy \UBC n.

  In order to finish the proof we only need to find the input
  manifolds~$N$ and~$M_k$ with the additional property that
  all~$M_k$ are simply connected (because then $M$ will also
  be simply connected). For example, we can take $N := S^n$ and
  $M_k$ to be $(S^2 \times S^{n-1})^{\# k}$ minus a small $(n+1)$-ball.
\end{proof}

However, it is not clear whether one can find aspherical sequences of
this type with amenable fundamental group. Therefore, the following
problem remains open:

\begin{question}
  Let $M$ be an aspherical space with amenable fundamental group and
  let $n \in \N$.  Does $C_*(M;\Z)$ satisfy \UBC n?
\end{question}

%%%%%%%%%%%%%%%%%%%%%%%%%%%%%%%%%%%%%%%%%%%%%%%%%%%%%%%%
\section{Application: integral foliated simplicial volume and glueings along tori}\label{sec:glueing}

As a sample application of the uniform boundary condition for the pa\-ram\-etrised
$\ell^1$-norm, we prove a simple additivity statement for integral foliated
simplicial volume and glueings along tori. 

%%%%%%%%%
\subsection{Integral foliated simplicial volume}\label{subsec:ifsvrel}

As a first step, we generalise the definition of integral foliated
simplicial volume~\cite{mschmidt} to the case of manifolds with
boundary. 

\begin{rem}[parametrised relative fundamental cycles]
  Let $(M,\partial M)$ be an oriented compact connected $n$-manifold
  with boundary; if $\partial M$ is non-empty, we will in addition assume that $\partial M$
  is connected and $\pi_1$-injective. Let $\pi \colon
  \ucov M \longrightarrow M$ be the universal covering of~$M$ and let
  $U \subset \pi^{-1}(\partial M)$ be a connected
  component of~$\pi^{-1}(\partial M)$. In particular, $\pi|_U \colon
  U\longrightarrow \partial M$ is a universal covering for~$\partial
  M$.
  
  Let $\Gamma := \pi_1(M)$ and let $\alpha$ be a standard $\Gamma$-space;
  then $\pi^{-1}(\partial M)$ is closed under the $\Gamma$-action on~$\ucov M$
  and thus we can consider the subcomplex
  \[ D_* := L^\infty(\alpha;\Z) \otimes_{\Z \Gamma} C_*\bigl(\pi^{-1}(\partial M);\Z\bigr)
  \]
  of~$C_*(M;\alpha)$.
  
  A chain~$c \in C_n(M;\alpha)$ \emph{represents~$[M,\partial M]^\alpha$}
  if there exist $b \in C_{n+1}(M;\alpha)$ and $d \in D_n$
  as well as an ordinary relative fundamental cycle~$c_\Z \in C_n(M;\Z)$
  (representing~$[M,\partial M]_{\Z}$)  
  such that
  \[ c = c_\Z + \partial b + d;
  \]
  here, we view~$C_*(M;\Z)$ as a subcomplex of~$C_*(M;\alpha)$ via the
  inclusion of constant functions. 
  In particular, in this situation, we have
  $\partial c \in D_{n-1}$; more precisely, we have the equality 
  \[ \partial c = \partial c_\Z +\partial d
  \]
  in~$D_{n-1}$. So, $\partial c$ is a cycle in~$D_*$.

  We will now explain how $\partial c$ can be interpreted a
  parametrised fundamental cycle of~$\partial M$: Let $\Gamma_0 :=
  \pi_1(\partial M) \subset \Gamma$ (the fundamental groups of~$M$ and
  $\partial M$ should be taken with respect to the same basepoint
  in~$\partial M$).  Then the restriction~$\alpha_0 :=
  \res^\Gamma_{\Gamma_0} \alpha$ of the $\Gamma$-action to the corresponding
  $\Gamma_0$-action is a standard $\Gamma_0$-space and we have the
  canonical (isometric) chain isomorphism
  \begin{align*}
    D_* & \cong L^\infty(\alpha;\Z) \otimes_{\Z \Gamma} C_*\bigl(\pi^{-1}(\partial M);\Z\bigr)
    \\
    & \cong L^\infty(\alpha;\Z) \otimes_{\Z \Gamma} \Z \Gamma \otimes_{\Z \Gamma_0} C_*(U;\Z)
    \\
    & \cong L^\infty(\alpha_0;\Z) \otimes_{\Z \Gamma_0} C_*(U;\Z)
    \\
    & = C_*(\partial M;\alpha_0).
  \end{align*}
  Hence, the equation~$\partial c = \partial c_\Z + \partial d$ holds
  in the complex~$C_*(\partial M;\alpha_0)$, which shows that $\partial c$
  represents~$[\partial M]^{\alpha_0}$ (because $\partial c_\Z$ is an integral
  fundamental cycle of~$\partial M$).
\end{rem}

\begin{defi}[integral foliated simplicial volume]
  Let $(M,\partial M)$ be an oriented compact connected $n$-manifold
  (with possibly empty boundary; if the boundary is non-empty, we will
  assume that $\partial M$ is connected and $\pi_1$-injective) and let
  $\Gamma := \pi_1(M)$.
  \begin{itemize}
  \item If $\alpha$ is a standard $\Gamma$-space,
    then we write
    \[ \ifsv {M,\partial M}^\alpha
       := \inf \bigl\{ |c|_1 \bigm| \text{$c \in C_n(M;\alpha)$ represents~$[M,\partial M]^\alpha$}\bigr\}.
    \]
    for the \emph{$\alpha$-parametrised simplicial volume
      of~$(M,\partial M)$}.
  \item The \emph{integral foliated simplicial volume of~$(M,\partial M)$} is
    then defined as
    \[ \ifsv {M,\partial M} := \inf_{\alpha \in P(\Gamma)} \ifsv {M,\partial M}^\alpha,
    \]
    where $P(\Gamma)$ denotes the set of all (isomorphism classes of)
    standard $\Gamma$-spaces.
  \end{itemize}
\end{defi}

%%%%
\subsection{Glueings along tori}

We can now formulate the glueing result: 

\begin{prop}[glueings along tori]\label{prop:glueing}
  Let $n \in \N_{>0}$ and let $(M_+,\partial M_+)$, $(M_-,\partial
  M_-)$ oriented compact connected $(n+1)$-manifolds whose boundary is
  $\pi_1$-injective and homeomorphic to the $n$-torus~$(S^1)^n$ and
  let $f \colon \partial M_+ \longrightarrow \partial M_-$ be an
  orientation-preserving homeomorphism. Let $M := M_+ \cup_f M_-$
  be the oriented closed connected $(n+1)$-manifold obtained by glueing~$M_+$
  and~$M_-$ along the boundary via~$f$ and
  let $\Gamma := \pi_1(M) \cong \pi_1(M_+) *_{\pi_1(f)} \pi_1(M_-)$.
  Furthermore, let $\alpha = \Gamma \actson (X,\mu)$ be an essentially free
  standard $\Gamma$-space with
  \[ \ifsv{M_+,\partial M_+}^{\res^\Gamma_{\pi_1(M_+)} \alpha} = 0
  \quad \text{and}\quad
     \ifsv{M_-,\partial M_-}^{\res^\Gamma_{\pi_1(M_-)} \alpha} = 0.
  \]
  Then~$\ifsv{M}^\alpha = 0$. In particular, $\ifsv M = 0$.   
\end{prop}

In this situation, we equip the glued manifold~$M = M_+ \cup_f M_-$
with the orientation inherited from the positive orientation on~$M_+$ and the
negative orientation on~$M_-$. 

\begin{proof}
  We prove~$\ifsv{M}^\alpha = 0$ using the uniform boundary condition
  for the parametrised $\ell^1$-norm on tori: Let $\varepsilon \in
  \R_{>0}$ and let $\alpha_+ := \res^\Gamma_{\pi_1(M_+)}$, $\alpha_-
  := \res^{\Gamma}_{\pi_1(M_-)}$ denote the restricted parameter
  spaces. By the hypothesis on the parametrised simplicial volumes of
  the components~$M_+$ and~$M_-$, there exist chains~$c_+ \in
  C_{n+1}(M_+; \alpha_+)$ and $c_- \in C_{n+1}(M_-;\alpha_-)$ that
  represent~$[M_+,\partial M_+]^{\alpha_+}$ and $[M_-,\partial
    M_-]^{\alpha_-}$, respectively, and that satisfy
  \[ |c_+|_1 \leq \varepsilon
     \quad\text{and}\quad
     |c_-|_1 \leq \varepsilon.
  \]
  Then
  \[ c_0 := \partial c_+ - \partial c_-
  \]
  is a null-homologous cycle in~$C_{n}(\partial M_-;\alpha_0)$, where $\alpha_0$
  denotes the restriction of~$\alpha$ to the subgroup~$\pi_1(\partial M_-)$. We
  view $M_+$ and $M_-$ in the canonical way as subspaces of~$M$ and use the
  identification via~$f$ to view both $\partial c_+$ and $\partial c_-$
  as chains on~$\partial M_-$. 
  By construction,
  \[ |c_0|_1 \leq (n+2) \cdot |c_+|_1 + (n+2) \cdot |c_-|_1
             \leq 2 \cdot (n+2) \cdot \varepsilon.
  \]
  In view of the uniform boundary condition on the torus~$\partial
  M_-$ (Theorem~\ref{mainthm:parsvab}), there exists a chain~$b \in C_{n+1}(\partial M_-;\alpha_0)$
  with
  \[ \partial b = c_0
     \quad\text{and}\quad
     |b|_1 \leq K \cdot |c_0|_1,
  \]
  where $K$ is a UBC-constant for~$C_n(\partial M_-;\alpha_0)$ (which is independent
  of $\varepsilon$, $c_+$, and $c_-$).
  Then
  \[ c := c_+ - c_- - b \in C_{n+1}(M;\alpha) 
  \]
  is a cycle representing~$[M]^\alpha$ and, by construction,
  \[      \ifsv M ^\alpha 
     \leq |c|_1
     \leq |c_+|_1 + |b|_1 + |c_-|_1
     \leq 2 \cdot \varepsilon + K \cdot 2 \cdot (n+2) \cdot \varepsilon.
  \]
  Taking the infimum over all~$\varepsilon \in \R_{>0}$ shows that $\ifsv M ^\alpha = 0$. 
\end{proof}

\begin{cor}
  Under the hypotheses of Proposition~\ref{prop:glueing}, if in
  addition $\Gamma$ is residually finite and $\alpha = \Gamma \actson \widehat \Gamma$ is the canonical action of~$\Gamma$ on its profinite
  completion, then
  \[ \stisv M = 0
  \]
	(see p.~\pageref{def:stisv} for the definition of~$\stisv M$).
\end{cor}
\begin{proof}
  Because $\Gamma$ is residually finite and finitely generated,
  $\alpha$ is an essentially free standard $\Gamma$-space
  and~\cite[Theorem~2.6]{flps}
  \[ \stisv M = \ifsv M ^{\Gamma \actson \widehat \Gamma}.
  \]
  By Proposition~\ref{prop:glueing}, $\ifsv M ^{\Gamma \actson \widehat \Gamma} = 0$,
  which proves the corollary.
\end{proof}

\begin{rem}[growth and gradient invariants]
  In particular, in these situations, we obtain corresponding vanishing
  results for homology growth and logarithmic homology torsion
  growth~\cite[Theorem~1.6]{flps} as well as for the rank
  gradient~\cite{loehrg}.
\end{rem}
  
\begin{rem}[multiple boundary components, self-glueings]
  In the same way as in Proposition~\ref{prop:glueing}, one can also treat
  glueings along disconnected boundaries where each glued component is a torus
  as well as self-glueings along torus boundary components.
\end{rem}

It would be desirable to also obtain additivity formulae for glueings
along amenable boundaries in the case of summands with non-zero
integral foliated simplicial volume (as in the case of ordinary
simplicial volume)~\cite{vbc,kuessner}. However, one further ingredient for
such additivity results is the so-called equivalence
theorem~\cite{vbc,bbfipp}.

\begin{question}
  Can the F\o lner filling technique be used to give direct proofs of
  the equivalence theorem (in the aspherical case)? Can such an
  argument be refined to lead to an equivalence theorem for integral
  foliated simplicial volume?
\end{question}

%%%%%%%%%%%%%%
\subsection{Concrete examples}

We will now give a concrete class of examples for such torus glueings,
leading to new vanishing results for integral foliated simplicial volume.

\begin{lem}\label{lem:crossproduct}
  Let $(N,\partial N)$ be an oriented compact connected manifold with
  connected boundary (which might be empty), let $k \in \N_{>0}$, and let $M :=
  N \times (S^1)^{k}$. If $\alpha$ is an
  essentially free standard $\pi_1(M)$-space, then
  \[ \ifsv{M, \partial M}^\alpha = 0 .
  \]
\end{lem}
\begin{proof}
  This is a relative version of the product inequality by Schmidt~\cite[Theorem~5.34]{mschmidt}.
	It suffices to consider the case~$k=1$, i.e., $M=N\times S^1$.
	We write~$\Gamma := \pi_1(M)$ and~$\Lambda \subset \Gamma$ for the subgroup corresponding to the $S^1$-factor.
	Moreover we write~$n:= \dim M$.
	
	Then~$\res^\Gamma_\Lambda \alpha$ is an essentially free standard $\Lambda$-space.
	For every~$\epsilon \in \R_{> 0}$ there exists a parametrised fundamental 
	cycle~$c\in C_1(S^1;\res^\Gamma_\Lambda \alpha)$ of~$S^1$ with
	\[ |c|_1 \leq \epsilon.
	\]
	This follows from an application of the Rokhlin lemma~\cite[Theorem~1.9]{flps}\cite[Proposition~5.30]{mschmidt}.
	Let~$c_N \in C_{n-1} (N;\Z)$ be a relative fundamental cycle of~$N$. Then the cross-product
	\[ c_M := c_N \times c \in C_n (M;\alpha)
	\]
	is a representative of~$[M,\partial M]^\alpha$ and
	\begin{align*}
	  |c_M|_1 &\leq \binom{n}{n-1} \cdot |c_N|_1 \cdot |c|_1
          %\\&
                \leq n \cdot |c_N|_1 \cdot \epsilon.
	\end{align*}
	Hence, $\ifsv{M, \partial M}^\alpha =0$.
\end{proof}

\begin{cor}
  Let $(N_+, \partial N_+)$ and $(N_-,\partial N_-)$ be oriented compact connected
  manifolds with boundary, let $n_+ := \dim N_+$ and $n_-:= \dim N_-$, and
  suppose that $\partial N_+$ and $\partial N_-$ are $\pi_1$-injective tori of dimension~$n_+$
  and $n_-$, respectively. Furthermore, let $n \in \N_{> \max(n_+, n_-)}$, and let
  \[ f \colon \partial N_+ \times (S^1)^{n - n_+}
  \longrightarrow
  \partial N_- \times (S^1)^{n - n_-}
  \]
  be an orientation-preserving homeomorphism. Then the glued
  manifold
  \[ M := \bigl(N_+ \times(S^1)^{n - n_+}\bigr) \cup_f
    \bigl(N_- \times(S^1)^{n - n_-}\bigr).
  \]
  satisfies
  \[ \ifsv M ^\alpha = 0
  \]
  for every essentially free standard $\pi_1(M)$-space~$\alpha$. In particular, $\ifsv M = 0$.
\end{cor}
\begin{proof}
  We write $M_+ := N_+ \times(S^1)^{n - n_+}$ and $M_- := N_- \times(S^1)^{n - n_-}$  
  as well as $\Gamma := \pi_1(M)$ for the fundamental group of~$M = M_+ \cup_f M_-$.
  
  %Then there exists an essentially free standard $\Gamma$-space (for
  %example, the Bernoulli shift).
  Let $\alpha$ be an essentially free standard $\Gamma$-space. By 
  $\pi_1$-injectivity of the boundary tori, the restricted parameter
  spaces~$\res^\Gamma_{\pi_1(M_+)} \alpha$ and
  $\res^\Gamma_{\pi_1(M_-)} \alpha$ are essentially free as well. Then Lemma~\ref{lem:crossproduct}
  shows that the hypotheses of Proposition~\ref{prop:glueing} are satisfied
  and hence~$\ifsv M ^\alpha = 0 $ and $\ifsv M = 0$.
\end{proof}

%%%%%%%%%%%%%%%%%%%%%%%%%%%%%%%%%%%%%%%%%%%%%%%%%%%%%%%%
\section{Vanishing of $\ell^1$-homology of amenable groups}\label{sec:l1homology}

We will now give an application to $\ell^1$-homology of groups.  If
$\Gamma$ is a group, we write~$C_*(\Gamma)$ for the standard simplicial $\R
\Gamma$-resolution of~$\R$~\cite[p.~18]{brown} and $C_*(\Gamma;\R) := \R \otimes_{\R
  \Gamma} C_*(\Gamma)$ for the associated chain complex. This chain
complex is a normed chain complex with respect to the $\ell^1$-norm
(given by the basis of all $(n+1)$-tuples in~$\Gamma$ whose $0$-th
vertex is~$1$). Taking the $\ell^1$-completion of~$C_*(\Gamma;\R)$
leads to the \emph{$\ell^1$-chain complex}~$C^{\ell^1}_*(\Gamma;\R)$
of~$\Gamma$~\cite{mm,loehl1}. Then
\emph{$\ell^1$-homology}~$H^{\ell^1}_*(\Gamma;\R)$ is defined as the
homology of~$C^{\ell^1}_*(\Gamma;\R)$.  Using the F\o lner filling
technique, we can reprove the following result of Matsumoto and
Morita~\cite{mm} -- without using
bounded cohomology:

\begin{thm}\label{thm:l1homology}
  Let $\Gamma$ be an amenable group and let $n \in \N_{>0}$. Then
  \[ H^{\ell^1}_n(\Gamma;\R) \cong 0.
  \]
\end{thm}

For simplicity, we will only consider the case of trivial
coefficients; analogous arguments apply to more general coefficients,
including coefficients of a more integral nature (as in the integral
foliated case). Moreover, one can also prove the same results
for aspherical spaces with amenable fundamental group. 

The proof of Theorem~\ref{thm:l1homology} consists of the following
steps: We first prove that the $\ell^1$-semi-norm on ordinary group
homology of amenable groups is trivial
(Proposition~\ref{prop:groupnorm0}), without using bounded cohomology
or multicomplexes. We then show that the image of ordinary group
homology in $\ell^1$-homology is uniformly trivial
(Proposition~\ref{prop:UBCcompletion}). In the final step, we
subdivide $\ell^1$-cycles into ordinary cycles and apply the previous
step.

\begin{prop}\label{prop:groupnorm0}
  Let $\Gamma$ be an amenable group and let $n \in \N_{>0}$. Then
  \[ \| \alpha\|_1 = 0
  \]
  for all~$\alpha \in H_n(\Gamma;\R)$. 
\end{prop}
\begin{proof}
  This can be proved with the F\o lner filling technique (analogous 
  to the proof of vanishing of integral foliated simplicial volume
  of aspherical manifolds with amenable fundamental group~\cite{flps}):
  Let $n \in \N_{>0}$ and let $c \in C_n(\Gamma;\R)$ be a cycle.

  \emph{Lifting step.} We can write~$c = \sum_{j=1}^m a_j \cdot
  [1,\gamma_{j,1}, \dots, \gamma_{j,n}]$ in reduced form; the
  chain
  \[ \widetilde c := \sum_{j=1}^m a_j \cdot (1,\gamma_{j,1}, \dots, \gamma_{j,n})
     \in C_n(\Gamma)
  \]
  is a lift of~$c$. Let 
  $S := \bigl\{ \gamma_{j,1} \bigm| j \in \{1,\dots,m\} \bigr\}
  $ 
  and let $(F_k)_{k \in \N}$ be a F\o lner sequence for the finitely
  generated amenable group~$\Lambda := \langle S\rangle_\Gamma \subset \Gamma$.
  By construction, the canonical projection of~$\widetilde c$
  to~$\R \otimes_{\R \Lambda} C_*(\Gamma)$ is a cycle. 

  \emph{Filling step.} Let $k \in \N$. Then $\partial(F_k \cdot
  \widetilde c) \in C_{n-1}(\Gamma)$ is a cycle and the coned off
  chain
  \[ \widetilde b_k := s\bigl(\partial (F_k \cdot \widetilde c)\bigr) 
     \in C_n(\Gamma),
  \]
  where $s$ is given by~$(\eta_1,\dots, \eta_n) \mapsto (1,\eta_1,\dots,\eta_n)$,     
  satisfies
  \[ \partial \widetilde b_k = \partial (F_k \cdot \widetilde c) 
     \quad\text{and}\quad
     |\widetilde b_k|_1 \leq \bigl|\partial(F_k \cdot \widetilde c)\bigr|_1
  \]
	(because of $\partial \circ s = \id - s\circ \partial$).
  The same argument as in the lifting lemma (Lemma~\ref{lem:lifting}) shows
  that
  \[ \lim_{k \rightarrow \infty} \frac1{|F_k|} \cdot |\widetilde b_k|_1
     = \lim_{k \rightarrow \infty} \frac1{|F_k|}
       \cdot \bigl|\partial (F_k \cdot \widetilde c)\bigr|_1 
     = 0.
  \]
  By construction, $\widetilde b_k - F_k \cdot \widetilde c_k \in C_n(\Gamma)$
  is a cycle, whence null-homologous (the chain complex~$C_*(\Gamma)$ is
  contractible). 
  
  \emph{Quotient step.} Therefore, the chain
  \[ c_k := \frac1{|F_k|} \cdot \text{(canonical projection of~$\widetilde b_k$)}
     \in C_n(\Gamma;\R)
  \]
  is a cycle with~$[c_k] = 1/|F_k| \cdot |F_k| \cdot [c] = [c]$
  and~$|c_k|_1 \leq 1 / |F_k| \cdot |\widetilde b_k|_1$.  Hence,
  $\lim_{k \rightarrow \infty} |c_k|_1 = 0$ and so $\|[c]\|_1= 0$.
\end{proof}

\begin{prop}\label{prop:UBCcompletion}
  Let $C_*$ be a normed $\R$-chain complex, let
  $\overline C_*$ be its completion, and let $n \in \N$. Furthermore,
  we assume that the induced semi-norm~$\|\cdot\|$ on~$H_n(C_*)$ is
  trivial and that $C_*$ satisfies \UBC n. Then the map
  \[ H_n(i) \colon H_n(C_*) \longrightarrow H_n(\overline C_*)
  \]
  induced by the inclusion~$C_* \hookrightarrow \overline C_*$ is trivial.
  More precisely: There exists a constant~$K \in \R_{>0}$ with the following
  property: For every cycle~$c \in C_n$ there exists a chain~$b \in \overline C_{n+1}$
  with
  \[ \partial b = c
  \quad
  \text{and}
  \quad
  |b|_1 \leq K \cdot |c|_1.
  \]
\end{prop}
\begin{proof}
  Let $K \in \R_{>0}$ be an \UBC n-constant for~$C_*$  
  and let $c \in C_n$ be a cycle. By hypothesis, $\|[c]\| = 0$. Hence,
  there is a sequence~$(c_k)_{k \in \N}$ of cycles in~$C_n$ such that
  \[ \partial c_k = 0,
     \qquad
     |c_k| \leq \frac 1 {2^k} \cdot |c|,
     \qquad
     [c_k] = [c] \in H_n(C_*)
  \]
  holds for all~$k \in \N$; moreover, we will take~$c_0 := 0$.

  In view of \UBC n, there exists a sequence~$(b_k)_{k \in \N}$ in~$C_{n+1}$
  such that for all~$k \in \N$ we have
  \[ \partial b_k = c_{k+1} - c_k
     \quad\text{and}
     \quad
     |b_k| \leq K \cdot |c_{k+1} - c_k| \leq K \cdot \frac1{2^{k-1}} \cdot |c|.
  \]
  Then $b := \sum_{k=0}^\infty b_k$ is  a well-defined chain in~$\overline C_{n+1}$
  and one calculates (using continuity of the boundary operator and
  absolute convergence of all involved series)
  \[ \partial b = \sum_{k=0}^\infty \partial b_k = \sum_{k = 0}^\infty (c_{k+1} - c_k) = c_0 = c
  \]
  as well as
  \[ |b| \leq \sum_{k=0}^\infty |b_k| \leq K \cdot |c| \cdot \sum_{k=0}^\infty \frac1{2^{k-1}}
         \leq 4 \cdot K \cdot |c|.
  \]
  Therefore, the constant $K/4$ has the desired property.       
\end{proof}

\begin{proof}[Proof of Theorem~\ref{thm:l1homology}]
  Because $\Gamma$ is amenable, the chain
  complex~$C_*(\Gamma;\R)$ satisfies the uniform boundary condition in
  each degree (the proof of Proposition~\ref{mainprop:Q} easily
  adapts to the group case and $\R$-coefficients). Let $n\in \N_{>0}$ 
  and
  let $K' \in \R_{>0}$ be an \UBC{(n-1)}-constant
  for~$C_*(\Gamma;\R)$.
  
  We now consider a cycle~$c \in C^{\ell^1}_n(\Gamma;\R)$; 
  the goal is to find an $\ell^1$-chain~$b$ with~$\partial b = c$. 
  As a
  first step, we show that $c$ can be decomposed into an $\ell^1$-sum
  \[ c = \sum_{k=0}^\infty c_k
  \]
  of ordinary cycles~$c_k \in C_n(\Gamma;\R)$: 
  %with $\sum_{k=0}^\infty |c_k|_1 \leq 2 \cdot |c|_1$.
  %\clcomm{correct constant?!} 
  By definition of the $\ell^1$-norm, there is a sequence~$(z_k)_{k \in \N}$
  in~$C_n(\Gamma;\R)$ with
  \[ c = \sum_{k=0}^\infty z_k 
     \quad\text{and}\quad
     |c|_1 \leq \sum_{k=0}^\infty |z_k|_1.
  \]
  For~$k \in \N$ we consider the partial sum
  $s_k := \sum_{j=0}^k z_j \in C_n(\Gamma;\R).
  $ 
  Because $c$ is a cycle, $\lim_{k\rightarrow \infty} |\partial s_k|_1
  = |\partial c|_1 = 0$. Hence, by regrouping our sequence~$(z_k)_{k
    \in \N}$, we may assume without loss of generality that the
  sequences~$(|\partial s_k|_1)_{k \in \N}$ and $(z_k)_{k \in \N}$
  both are~$\ell^1$. By \UBC {(n-1)}, there exists a sequence~$(w_k)_{k
    \in \N}$ in $C_n(\Gamma;\R)$ such that
  \[ \partial w_k = \partial s_k
     \quad\text{and}\quad
     |w_k|_1 \leq K' \cdot |\partial s_k|_1
  \]
  holds for all~$k \in \N$. For $k \in \N$ we set
  $ c_k := z_k - w_k + w_{k-1} 
  $ 
  (where $w_{-1} := 0$ and $s_{-1} := 0$). Then $\partial c_k = 0$ and
  $ |c_k|_1 \leq |z_k|_1 + K' \cdot |\partial s_{k}|_1 + K' \cdot |\partial s_{k-1}|_1.
  $ 
  Hence, $c = \sum_{k=0}^\infty c_k$ is an $\ell^1$-sum of ordinary cycles in~$C_n(\Gamma;\R)$.

  We will now apply Proposition~\ref{prop:UBCcompletion} to~$C_*(\Gamma;\R)$.
  In view of Proposition~\ref{prop:groupnorm0}, the $\ell^1$-semi-norm
  on~$H_n(\Gamma;\R)$ is trivial. Moreover, $C_*(\Gamma;\R)$
  satisfies \UBC n (see above). So Proposition~\ref{prop:UBCcompletion} indeed
  can be applied. Hence, there is a~$K \in \R_{>0}$ such that
  for each~$k \in \N$ there exists a~$b_k \in C^{\ell^1}_{n+1}(\Gamma;\R)$ with
  \[ \partial b_k = c_k
     \quad\text{and}\quad
     |b_k|_1 \leq K \cdot |c_k|_1.
  \]
  Therefore,
  \[ b := \sum_{k=0}^\infty b_k \in C^{\ell^1}_{n+1}(\Gamma;\R)
  \]
  is a well-defined $\ell^1$-chain that satisfies
  $\partial b = \sum_{k = 0}^\infty \partial b_k
  = \sum_{k = 0}^\infty c_k = c.
    %\qedhere
  $
  %Moreover, 
  %\[ |b|_1 \leq \sum_{k=0}^\infty |b_k|_1 \leq K \cdot \sum_{k=0}^\infty |c_k|_1 \leq 2 \cdot K \cdot |c|_1,
  %\]
  %which is a refined version of vanishing of~$H^{\ell^1}_n(\Gamma;A)$.
\end{proof}

% formulate generalisations for more integral coefficients?!

%%%%%%%%%%%%%%%%%%%%%%%%%%%%%%%%%%%%%%%%%%%%%%%%%%%%%%%%%%%%%%%%

\medskip
\vfill

\noindent
\emph{Daniel Fauser}\\
\emph{Clara L\"oh}\\[.5em]
  {\small
  \begin{tabular}{@{\qquad}l}
    Fakult\"at f\"ur Mathematik, 
    Universit\"at Regensburg, 
    93040 Regensburg\\
    %Germany\\
    \textsf{daniel.fauser@mathematik.uni-r.de},  
    \textsf{http://www.mathematik.uni-r.de/fauser}\\
    \textsf{clara.loeh@mathematik.uni-r.de}, 
    \textsf{http://www.mathematik.uni-r.de/loeh}
  \end{tabular}}
\end{document}